\documentclass[a4paper, 11pt, english]{article}
\usepackage[utf8]{inputenc}
\usepackage[T1]{fontenc}
\usepackage{babel}
\usepackage{graphicx,xcolor,soul}
\usepackage{stmaryrd}
\usepackage[a4paper]{geometry}
\usepackage{amsmath,amsfonts,amssymb,amsthm,epsfig,epstopdf,url,array}
\usepackage{rotating}
\usepackage[colorlinks=true,citecolor=red,linkcolor=blue,pdfpagetransition=Blinds]{hyperref}
\usepackage{cleveref}
\usepackage{nameref}
\usepackage{enumitem}
\usepackage{comment}
\Crefname{paragraph}{Section}{Sections}
\usepackage{fancyhdr}
\usepackage{dsfont}
\usepackage{todonotes}

\usepackage{fullpage}

\newcommand{\ensemblenombre}[1]{\mathbb{#1}}
\newcommand{\N}{\ensemblenombre{N}}

\newcommand{\R}{} 
\renewcommand{\R}{\ensemblenombre{R}}
\newcommand{\C}{\ensemblenombre{C}}

\newcommand{\T}{\ensemblenombre{T}}

\newcommand{\dive}[1]{\mathrm{div}}

\theoremstyle{plain} 
\newtheorem{pr}{Proposition}[section] 
\newtheorem{tm}[pr]{Theorem}
\newtheorem{lm}[pr]{Lemma}

\newtheorem{cl}[pr]{Claim}
\theoremstyle{definition}

\newtheorem{rmk}[pr]{Remark}

\newtheorem{ass}[pr]{Assumption}

\numberwithin{equation}{section}

\makeatletter
\let\original@addcontentsline\addcontentsline
\newcommand{\dummy@addcontentsline}[3]{}
\newcommand{\DeactivateToc}{\let\addcontentsline\dummy@addcontentsline}
\newcommand{\ActivateToc}{\let\addcontentsline\original@addcontentsline}
\makeatother

\begin{document}

\author{ Ludovick Gagnon\footnote{Universit\'e de Lorraine, CNRS, Inria, IECL, F-54000 Nancy, France. E-mail : \tt{ludovick.gagnon@inria.fr}} \qquad  Kévin Le Balc'h\footnote{Inria, Sorbonne Université, CNRS, Laboratoire Jacques-Louis Lions, Paris, France. {\tt kevin.le-balc-h@inria.fr} \quad -- \quad Both authors are partially supported by the Project TRECOS ANR-20-CE40-0009 funded by the ANR (2021--2024).}}

\title{Control of blow-up profiles for the mass-critical focusing nonlinear Schrödinger equation on bounded domains}

\maketitle

\begin{abstract}
In this paper, we consider the mass-critical focusing nonlinear Schrödinger on bounded two-dimensional domains with Dirichlet boundary conditions. In the absence of control, it is well-known that free solutions starting from initial data sufficiently large can blow-up. More precisely, given a finite number of points, there exists particular profiles blowing up exactly at these points at the blow-up time. For pertubations of these profiles, we show that, with the help of an appropriate nonlinear feedback law located in an open set containing the blow-up points, the blow-up can be prevented from happening. More specifically, we construct a small-time control acting just before the blow-up time. The solution may then be extended globally in time. This is the first result of control for blow-up profiles for nonlinear Schrödinger type equations. Assuming further a geometrical control condition on the support of the control, we are able to prove a null-controllability result for such blow-up profiles. Finally, we discuss possible extensions to three-dimensional domains.
\end{abstract}

\tableofcontents

\section{Introduction}

\subsection{The mass-critical focusing nonlinear Schrödinger equation and blow-up}

Let $T>0$ and $\Omega$ be a smooth, bounded domain of $\R^2$. We study the $L^2$-critical focusing nonlinear Schrödinger equation in $\Omega$, with Dirichlet boundary conditions
\begin{equation}
	\label{eq:L2CriticalNLS}
		\left\{
			\begin{array}{ll}
				 i \partial_t \psi + \Delta \psi = - |\psi|^{2} \psi  & \text{ in }  (0,T) \times \Omega, 
				\\
				\psi = 0 & \text{ on } (0,T)\times \partial \Omega, 
				\\
				\psi(0, \cdot) = \psi_0 & \text{ in } \Omega.
			\end{array}
		\right.
\end{equation}
Note that \eqref{eq:L2CriticalNLS} is the focusing nonlinear cubic Schrödinger equation, modelling for instance nonlinear optics for the self-focusing of intense laser beams in hollow core fibers. We refer to \cite{SS99} for further references of physical models that are described by \eqref{eq:L2CriticalNLS}.

By the Sobolev embedding $H^2(\Omega) \hookrightarrow L^{\infty}(\Omega)$, it is classical to prove that \eqref{eq:L2CriticalNLS} is locally well-posed in $H^2(\Omega) \cap H_0^1(\Omega)$. 
\begin{tm}[Lemma 2.1 of \cite{BGT03}]
\label{lm:localwpH2H01}
For any $R>0$, there exists a time $T>0$ such that for every $\psi_0 \in B_R := \{ \psi_0 \in H^2(\Omega) \cap H_0^1(\Omega)\ ;\ \|\psi_0\|_{H^2(\Omega)} \leq R\},$
the problem \eqref{eq:L2CriticalNLS} has a unique solution in $C([0,T];H^2(\Omega)\cap H^1_0(\Omega))$. In addition, the flow map $\psi_0 \to \psi$ is uniformly continuous from $B_R$ to $C([0,T];H^2(\Omega)\cap H^1_0(\Omega))$. 
\end{tm}
Moreover,  for such solutions, the following conservation laws hold for every $t \in [0,T]$,
\begin{align}
\|\psi(t,\cdot)\|_{L^2(\Omega)} &= \|\psi_0\|_{L^2(\Omega)}, \label{eq:consevationmass} \\
E(\psi(t, \cdot)) := \frac{1}{2} \|\nabla \psi(t,\cdot)\|_{L^2(\Omega)}^2 - \frac{1}{4} \|\psi(t,\cdot)\|_{L^{4}(\Omega)}^{4}
&= \frac{1}{2} \|\nabla \psi_0\|_{L^2(\Omega)}^2 - \frac{1}{4} \|\psi_0\|_{L^{4}(\Omega)}^{4} = E(\psi_0).\label{eq:consevationenergy}
\end{align} 
For every $\psi_0 \in H^2(\Omega) \cap H_0^1(\Omega)$, the maximal time of existence $T_{\max}(\psi_0) \in (0,+\infty]$ of the maximal solution $\psi \in C([0,T_{\max}(\psi_0));H^2(\Omega) \cap H_0^1(\Omega))$ of \eqref{eq:L2CriticalNLS} can be defined. Furthermore, we have the following blow-up criterion \cite{Caz03},
\begin{equation}
T_{\max}(\psi_0) < +\infty \Rightarrow \|\psi(t,\cdot)\|_{H^2(\Omega)} \to +\infty\ \text{as}\ t \to T_{\max}(\psi_0)^-.
\end{equation}

Note that the local well-posedness of \eqref{eq:L2CriticalNLS} can be generalized to $H_0^1(\Omega)$.
\begin{tm}[Theorem 3.6.1 of \cite{Caz03}]
\label{lm:localwpH01}
For any $R>0$, there exists a time $T>0$ such that for every $\psi_0 \in B_R := \{ \psi_0 \in  H_0^1(\Omega)\ ;\ \|\psi_0\|_{H_0^1(\Omega)} \leq R\},$
the problem \eqref{eq:L2CriticalNLS} has a unique solution in $C([0,T];H^1_0(\Omega))$. 
\end{tm}
However the uniform continuity of the flow map is not clear in this case. The conservation laws \eqref{eq:consevationmass} and \eqref{eq:consevationenergy} still hold. Moreover the maximal time of existence $T_{\max}^{H^1_0}(\psi_0) \in (0,+\infty)$ can still be defined and then by using \eqref{eq:consevationmass} the blow-up criterion takes the following form
\begin{equation}
T_{\max}^{H^1_0}(\psi_0) < +\infty \Rightarrow \|\nabla \psi(t,\cdot)\|_{L^2\Omega)} \to +\infty\ \text{as}\ t \to T_{\max}^{H^1_0}(\psi_0)^-.
\end{equation}

Blow-up can indeed occur. Following \cite{Kav87}, assume for instance that $\Omega$ is a star-shaped domain with respect to $0 \in \Omega$ and let $\psi_0 \in H^2(\Omega) \cap H_0^1(\Omega)$ such that $E(\psi_0) < 0$. Set $V(t) = \int_{\Omega} |x|^2 |\psi(t,x)|^2 dx$ then we have the following virial identity
\begin{equation}
\frac{1}{16} \frac{d^2}{dt^2} V(t) = E(\psi(t,\cdot)) - \frac{1}{4} \int_{\partial \Omega} |\nabla \psi \cdot n|^2 x \cdot n dx
\leq E(\psi_0) < 0
\end{equation}
where $n$ is the outer normal to $\partial \Omega$. Since $V(t)$ is a positive concave function, the solution blows up in finite time.

Blow-up solutions can be constructed from the ground state solitary wave. In $\R^d$, the ground state solitary wave $Q$ is the unique positive radially symmetric solution of
\begin{equation}
\label{eq:defQ}
- \Delta Q + Q = |Q|^{p-1} Q \ \text{in}\ \R^d,
\end{equation}
where
\begin{equation}
\label{eq:defp}
p = 1+ \frac{4}{d}.
\end{equation}
According to \cite{Kwo89} and \cite{Wei83}, this solution decays exponentially at infinity, 
\begin{equation}
\label{eq:expdecayQ}
\forall \alpha \in \N^d,\ \exists C_\alpha, D_\alpha >0,\ \forall x \in \R^d,\ |\partial^{\alpha} Q(x)| \leq C_\alpha e^{- D_\alpha |x|}.
\end{equation}
 In $\R^d$, the blow-up regime in the mass-critical case 
\begin{equation}
	\label{eq:L2CriticalNLSRd}
		\left\{
			\begin{array}{ll}
				 i \partial_t \psi + \Delta \psi = - |\psi|^{p-1} \psi  & \text{ in }  (0,T) \times \R^d, 
				\\
				\psi(0, \cdot) = \psi_0 & \text{ in } \R^d,
			\end{array}
		\right.
\end{equation}
is rather well-understood for $\psi_0\in H^1(\R^d)$. Weinstein \cite{Wei83} proved that initial data satisfying $\|\psi_0\|_{L^2(\R^d)}< \|Q\|_{L^2(\R^d)}$ leads to the global existence in time of solutions, combining a sharp interpolation estimate to a Gagliardo-Nirenberg inequality. This bound have been proven to be sharp by Merle in $\Sigma := \{ \psi_0 \in H^1(\R^d) \, | \,  x\psi_0 \in L^2(\R^d) \}$ using the pseudo-conformal transformation on the ground state $Q$ \cite{Mer93}. Indeed, there exists an initial data $\psi_0 \in \Sigma$ such that $\|\psi_0 \|_{L^2(\R^d)}= \|Q\|_{L^2(\R^d)}$ and such that the corresponding solution $\psi$ to \eqref{eq:L2CriticalNLSRd} blow-up as $\|\nabla \psi(t,.)\|_{L^2(\R^d)} \simeq 1/(T-t)$ as $t\rightarrow T^{-}$. The blow-up phenomenon of \eqref{eq:L2CriticalNLSRd} was further characterized \cite{MR03, MR04, MR05a, MR05b, MR06, Rap05} in the set 
\[
\mathcal{B}_{\alpha^*}=\left\{ \psi_0 \in H^1(\R^d)\, \left| \, \|Q\|_{L^2}^2 < \|\psi_0\|_{L^2}^2 < \|Q\|_{L^2}^2 + \alpha^* \right. \right\},
\]
for $\alpha^*>0$ sufficiently small and $E(\psi_0)<0$, assuming the so-called \textit{spectral property} in high dimensions. For initial conditions in this set, there exists $T=T(\psi_0)>0$ such that the solutions writing as 
\begin{equation}\label{eq:profilblowupmasscri}
\psi(x,t)=\dfrac{e^{i\gamma(t)}}{\lambda(t)}(Q+\epsilon)\left(\frac{x-x(t)}{\lambda(t)},t\right),
\end{equation}
with $\|\epsilon(.,t)\|_{H^1(\R^d)}\rightarrow 0$ as $\alpha^* \rightarrow 0$, blow up  at the rate, 
\begin{equation}\label{eq:loglog}
\| \nabla \psi(.,t)\|_{L^2(\R^d)} \simeq \dfrac{1}{\lambda(t)}  \simeq \sqrt{\dfrac{\log|\log(T-t)|}{T-t}} , \quad t\rightarrow T^-.
\end{equation}
For bounded domains $\Omega\subset \R^d$, two types of blow-up are found in the literature for the $L^2$ mass-critical case. If $d\leq 4$, a log-log type of blow-up (as \eqref{eq:loglog}) is exhibited in the set $\mathcal{B}_{\alpha^*}$ by Planchon and Rapha\"el \cite{PR07}.  The second type of blow-up is a result from Godet \cite[Theorem 1]{God11}, a generalization of \cite{BGT03} of a single point blow-up in dimension $2$ to the case of a finite number blowing-up points in dimension $2$ or $3$. It can also be considered as a generalization of the results from \cite{Mer90} established in $\R^d$ to the case of bounded domains. We only state it in $2$-d for our purpose.

\begin{tm}[\cite{BGT03,God11}]
\label{thm:mainresultgod}
Let $\Omega$ be a smooth bounded domain of $\R^2$. Let $n \geq 1$ and $x_1, \dots,x_n$ be $n$ distinct points in $\Omega$. Then, there exist $c>0$, $\lambda_0 >0$ such that for every $\lambda \geq \lambda_0$, there exist a time $T_{\lambda} \in (0, c \lambda^{-2})$ and a solution $\phi_{\lambda} \in C([0,T_\lambda);H^2(\Omega)\cap H^1_0(\Omega))$ to \eqref{eq:L2CriticalNLS}, which blow-up at the points $x_1, \dots, x_n$ in time $T_\lambda$. Moreover, the following properties hold.
\begin{enumerate}
\item The mass is conserved,
\begin{equation}
\label{eq:consmassL2philammbda}
\|\phi_\lambda(t, \cdot)\|_{L^2(\Omega)}= \sqrt{n} \|Q\|_{L^2(\R^2)},\qquad \forall t \in [0,T_\lambda).
\end{equation}
\item The blow-up occurs in the energy space $H_0^1(\Omega)$ with the following speed $$\|\nabla \phi_{\lambda}(t,\cdot)\|_{L^2(\Omega)} \underset{t \to T_\lambda^-}{\sim}\frac{\sqrt{n} \| \nabla Q\|_{L^2(\R^2)}}{\lambda (T_{\lambda}-t)}.$$
\item The function $\phi_\lambda$ decomposes as
\begin{equation}
\label{eq:decompositionphilambda}
 \phi_{\lambda}(t,x)=R_{\lambda}(t,x)  + r_{\lambda}(t,x), \qquad \forall (t,x) \in [0,T_\lambda) \times \Omega,
\end{equation}
with $R_{\lambda}, r_{\lambda} \in C([0,T_{\lambda});H^2(\Omega) \cap H^1_0(\Omega))$ satisfying
 \begin{equation}
 \label{eq:defRLambda}
R_{\lambda}(t,x) = \frac{1}{\lambda ^{} (T_{\lambda}-t)^{}} \sum_{k=1}^n e^{\frac{i(4-\lambda^2 |x-x_k|^2)}{4 \lambda ^2(T_\lambda-t)}} \varphi_k(x) Q \left(\frac{x-x_k}{\lambda (T_{\lambda} -t)}\right),\ \forall (t,x) \in [0,T_\lambda) \times \Omega,
\end{equation}
where $\varphi_1, \dots, \varphi_n \in C^{\infty}_0(\Omega;[0,1])$ have disjoint supports, verify $\varphi_k=1$ near $x_k$  and 
\begin{equation}
\label{eq:estimationrlambda}
\|r_\lambda(t,\cdot)\|_{H^2(\Omega)} \leq ce^{-\frac{\kappa}{\lambda(T_\lambda-t) }}, \qquad \forall t\in [0,T_\lambda)\qquad (\kappa>0). 
\end{equation}
\end{enumerate}
\end{tm}
For $n=1$, i.e. considering the case of a single blow-up point, the blow-up solutions of \Cref{thm:mainresultgod} in the energy space $H_0^1(\Omega)$ are actually of minimal $L^2$-norm as noticed in \cite{PR07}. This is a consequence of the conservation laws \eqref{eq:consevationmass}, \eqref{eq:consevationenergy} and the sharp Gagliardo-Nirenberg inequality stated here in dimension $d=2$, see \cite{Wei83}, 
\begin{equation}
\displaystyle E(\psi)  \geq \frac{1}{2} \left(\int_{\Omega} |\nabla \psi|^2\right)\left(1 - \left( \dfrac{\|\psi\|_{L^2(\Omega)}}{\|Q\|_{L^2(\R^2)}} \right)^{2} \right).
\end{equation}
Bounds on the blow-up rate for minimal mass blowing-up solutions are studied in \cite{Ban04}. Furthermore, by using Brezis-Gallouet's inequality (see \cite{BG80}), the global well-posedness in the energy space $H_0^1(\Omega)$ implies the global well-posedness in $H^2(\Omega) \cap H_0^1(\Omega)$. We then have the following result.
\begin{tm}[\cite{PR07}, \cite{BG80}]
\label{tm:globalwp}
The following two properties hold.
\begin{enumerate}
\item  For every initial data $\psi_0 \in H_0^1(\Omega)$ satisfying $\|\psi_0\|_{L^2(\Omega)} < \|Q\|_{L^2(\R^2)}$, \eqref{eq:L2CriticalNLS} is globally well-posed in the energy space $H_0^1(\Omega)$.
\item For every initial data $\psi_0 \in H^2(\Omega) \cap H_0^1(\Omega)$ satisfying $\|\psi_0\|_{L^2(\Omega)} < \|Q\|_{L^2(\R^2)}$, \eqref{eq:L2CriticalNLS} is globally well-posed in the functional space $H^2(\Omega) \cap H_0^1(\Omega)$.
\end{enumerate}
\end{tm}

The goal of this article is to see if, by acting locally on the equation \eqref{eq:L2CriticalNLS}, one can prevent the blow-up from happening for the particular blow-up profiles of \Cref{thm:mainresultgod}. More precisely, the controlled nonlinear Schrödinger equation that we consider in the following is
\begin{equation}
	\label{eq:L2CriticalNLSControl}
		\left\{
			\begin{array}{ll}
				 i \partial_t \psi + \Delta \psi = - |\psi|^{2} \psi  + v \mathds{1}_{\omega} & \text{ in }  (0,T) \times \Omega, 
				\\
				\psi = 0 & \text{ on } (0,T)\times \partial \Omega, 
				\\
				\psi(0, \cdot) = \psi_0 & \text{ in } \Omega.
			\end{array}
		\right.
\end{equation}
In \eqref{eq:L2CriticalNLSControl}, at time $t \in [0,T]$, $\psi(t,\cdot) : \Omega \to \C$ is the state and $v(t,\cdot) : \omega \to \C$ is the control.

\subsection{Control of the mass-critical focusing nonlinear Schrödinger equation}

Our first main result shows that, with the help of a localized feedback control whose support contains the blow-up points, one can prevent the blow-up from happening and then build a global solution to the controlled problem \eqref{eq:L2CriticalNLSControl}.

\begin{tm}\label{tm:mainresult1}
Let $\Omega$ be a smooth bounded domain in $\R^2$. Let $n \geq 1$ and $x_1, \dots, x_n$ be $n$ distinct points in $\Omega$. Let $\omega \subset \Omega$ be an open set such that $x_1, \dots, x_n \in \omega$. Then, there exist $\delta>0$, $\lambda_0 >0$ such that for every $\lambda \geq \lambda_0$, there exists a feedback control $v_{\lambda} \in L^{\infty}([0,\infty);H^2(\Omega) \cap H_0^1(\Omega))$ supported in $(0,T_{\lambda}) \times \omega$ and an associated global solution $\psi \in C([0,+\infty);H^2(\Omega) \cap H_0^1(\Omega))$ of \eqref{eq:L2CriticalNLSControl} starting from $\psi(0, \cdot) = \phi_{\lambda}(0,\cdot) + w_0$, where $\phi_\lambda$ is the blow-up profile from Theorem \ref{thm:mainresultgod} and $\|w_0\|_{H^2\cap H^1_0 (\Omega)} < \delta$. 
\end{tm}

Our second main result shows that one can even achieve null controllability for such blow-up profiles, assuming a control condition that we present now. The linearized control system of \eqref{eq:L2CriticalNLSControl} around the equilibrium $(\psi, v) = (0,0)$ is given by the linear Schrödinger equation
\begin{equation}
\label{eq:linschro}
\begin{cases}
i\partial_t \psi  + \Delta \psi  = v \mathds{1}_\omega, \quad  & \mathrm{in}\ (0,T) \times \Omega,  \\
\psi=0, \quad & \mathrm{on}\ (0,T) \times \partial \Omega, \\
\psi(0,\cdot)=\psi_0, \quad & \mathrm{in}\ \Omega.
\end{cases}
\end{equation}

Let us introduce the following hypothesis: 
\begin{ass}
\label{ass:linschroControllable}
The linear Schrödinger equation \eqref{eq:linschro} is small-time null-controllable (STNC). That is, for every $T>0$ and for every $\psi_0 \in L^2(\Omega)$, there exists $v \in L^2([0,T];L^2(\omega))$ such that the solution $\psi \in C([0,T];L^2(\Omega))$ of \eqref{eq:linschro} satisfies $\psi(T, \cdot) = 0$.
\end{ass}
Controllability for the linear Schrödinger equation \eqref{eq:linschro} began to be extensively investigated in the 1990s. Up to now, a necessary and sufficient geometrical condition on $\omega$, depending on $\Omega$, for the (STNC) of \eqref{eq:linschro} to hold has not been identified. Let us present some of the well-known results in this direction (see \cite{Lau14} for a survey up to 2014). For a general smooth bounded domain $\Omega \subset \R^d$, $d \geq 1$, \cite{Leb92} guarantees that the so-called Geometric Control Condition (GCC) is sufficient for the (STNC) of the Schrödinger equation to hold. The proof of this result is based on microlocal analysis. (GCC) can be roughly formulated as follows: the subdomain $\omega$ is said to
satisfy (GCC) in time $T >0$ if and only if all rays of geometric optics propagating inside the domain $\Omega$ and bouncing off the boundary reach the control set $\omega$ non-diffractively in time less than $T$. This condition is not necessary for some specific domains. For instance, for $\Omega = (0,l_1) \times \dots \times (0,l_d)$, any nonempty open subset $\omega$ is sufficient for the STNC of \eqref{eq:linschro}; see \cite{Jaf90} or \cite{KL05} for a proof using Ingham's estimates, or \cite{BBZ13}, \cite{AM14} for a proof using semiclassical measures in the case of the torus. For the unit disk $\mathbb D \subset \R^2$, explicit eigenfunctions concentrate near the boundary, so one can prove that exact controllability holds if and only if $\overline{\omega}$ contains a (small) part of the boundary $\partial \mathbb D$ (see \cite{ALM16}).\\

Under Assumption \ref{ass:linschroControllable}, we are able to prove small-time local null controllability in $H^2(\Omega) \cap H_0^1(\Omega)$ of \eqref{eq:L2CriticalNLSControl} (see \Cref{thm:localnullcontrollabilityNLS} below), leading to our second main result.
\begin{tm}
\label{tm:mainresult2}
Let $\Omega$ be a smooth bounded domain in $\R^2$. Let $n \geq 1$ and $x_1, \dots, x_n$ be $n$ distinct points in $\Omega$. Let $\omega \subset \Omega$ be an open set such that $x_1, \dots, x_n \in \omega$ and such that Assumption \ref{ass:linschroControllable} holds. Let $T>0$. Then, there exist $\delta>0$, $\lambda_0 >0$ such that for every $\lambda \geq \lambda_0$, there exists a control $v_\lambda \in L^{\infty}([0,T];H^2(\Omega) \cap H_0^1(\Omega))$ supported in $(0,T) \times \omega$ and an associated solution $\psi \in C([0,T];H^2(\Omega) \cap H_0^1(\Omega))$ of \eqref{eq:L2CriticalNLSControl}, starting from $\psi(0, \cdot) = \phi_{\lambda}(0,\cdot)+w_0$, where $\phi_\lambda$ is the blow-up profile of Theorem \ref{thm:mainresultgod} and $\|w_0\|_{H^2\cap H^1_0(\Omega)}<\delta$, satisfying 
\begin{equation}
\label{eq:psizero}
\psi(T,\cdot) = 0.
\end{equation}
\end{tm}

\subsection{Comments and generalizations}

The following remarks are worth noting.
\begin{itemize}
\item \textbf{Regularity of the controlled trajectory.} Note that the regularity of the control $v$ in both \Cref{tm:mainresult1} and \Cref{tm:mainresult2} does not immediately imply that the associated solution $\psi$ belongs to $C([0,\infty);H^2(\Omega) \cap H_0^1(\Omega))$ in the first case or to $C([0,T];H^2(\Omega) \cap H_0^1(\Omega))$ in the second case. This regularity for the trajectory will, however, follow from our construction.
\item \textbf{Robustness with respect to perturbation of the feedback control.} Let us emphasize that the control $v_{\lambda} \in L^{2}([0,\infty);H^2(\Omega) \cap H_0^1(\Omega))$ supported in $(0,T_{\lambda}) \times \omega$ of Theorem \ref{tm:mainresult1} is a piecewise continuous nonlinear closed-loop control. Such a control strategy is relevant for practical applications and is robust with respect to perturbations, which is why we focus on this case. The design of a simple but a priori not stable open-loop control and an associated global solution is performed in Section \ref{sec:openloopcontrol} below. We underline here that one cannot extend the conclusions of Theorem \ref{tm:mainresult1} and Theorem \ref{tm:mainresult2} from the initial data $\phi_{\lambda}(0,\cdot)$ to the initial data $\phi_{\lambda}(0,\cdot) + w_0$ with $w_0 \neq 0$ using a simple perturbation argument, such as Gronwall's inequality. Indeed, these perturbation arguments do not guarantee that the solution starting from a perturbation of $\phi_{\lambda}(0,\cdot)$ remains close to $\phi_{\lambda}(t,\cdot)$ for $t > 0$. The very fact that the time of blow-up for perturbed solutions is the same as that of $\phi_{\lambda}(t,\cdot)$ is also not a simple consequence of Theorem \ref{thm:mainresultgod}. Last but not least, we finally highlight that the size $\delta$ of the allowed perturbations of the blow-up profiles does not depend on $\lambda$. By looking at the details of the proof of Proposition \ref{pr:rapidstab} below, one could notice that the parameter $\delta$ in Theorem \ref{tm:mainresult1}, respectively in Theorem \ref{tm:mainresult2}, can be chosen of the following form $\delta = c \|Q\|_{L^2(\R^2)}$, respectively $\delta = c \delta_{T/2}$, where $c = 1/16$ and $\delta_{T/2}$ is the radius of the ball centered at $0$ in $H^2(\Omega) \cap H_0^1(\Omega)$ where the local null-controllability of \eqref{eq:L2CriticalNLSControl} holds at time $T/2$. We could probably increase the value of the constant $c > 0$, but we do not investigate this issue.

\item \textbf{Blow-up versus localized control.} To the best of our knowledge, \Cref{tm:mainresult1} is the first result demonstrating the possibility of preventing blow-up in a nonlinear Schrödinger equation using localized control. From \eqref{eq:consmassL2philammbda}, it is clear that $\|\psi(0,\cdot)\|_{L^2(\Omega)} = \sqrt{n}\|Q\|_{L^2(\R^2)}$ is not small, meaning \Cref{tm:mainresult2} cannot be directly derived from a local null-controllability result. Instead, we focus on the specific knowledge of the blow-up profiles and use control regions that must include the blow-up points. Notice that even for semilinear heat equations with power-like nonlinearities, where blow-up solutions at prescribed points can also be constructed \cite{Mer92}, establishing such a property remains an open problem. Furthermore, our results are in stark contrast to \cite{LZ22}, where the authors showed that a feedback control acting on $\omega$ could induce a solution of the heat equation to blow up at time $T$ at a unique point $a \in \omega$. Similarly, our results differ significantly from \cite{FCZ00}, in the context of semilinear heat equations with weakly superlinear nonlinearities, where the authors were able to prevent blow-up and even drive the solution to zero for arbitrary times, with nonempty open sets control regions, and for any initial data. Whether such a property can be proven or disproven for nonlinear Schrödinger equations remains entirely open. A related result to \Cref{tm:mainresult2} can be found in the recent work of the first author in \cite{Gag23}, where local controllability around ground state solitary waves with different scalings is achieved using a whole boundary control.
\item \textbf{Examples of geometrical control assumptions.} When $\Omega = (0,l_1) \times (0,l_2)$ is a rectangle, any nonempty open subset $\omega$ leads to the STNC of \eqref{eq:linschro}, so \Cref{tm:mainresult2} holds assuming only that $\omega$ contain the blow-up points of the profile $\phi_{\lambda}$. On the other hand, when $\Omega = \mathbb{D}$ the unit disk, \Cref{tm:mainresult2}  holds assuming that $\omega$ contain the blow-up points and $\overline{\omega} \cap \partial D \neq \emptyset$. When $\Omega$ is an arbitrary smooth bounded $2$-d domain, \Cref{tm:mainresult2} holds assuming that GCC is satisfied. See the previous discussion after \Cref{ass:linschroControllable}.
\item \textbf{Extensions.} The results of \Cref{tm:mainresult1} and \Cref{tm:mainresult2} can be extended without significant technical challenges to the 1-dimensional case, where the $L^2$-critical focusing nonlinearity in \eqref{eq:L2CriticalNLSRd} is considered with $p$ as defined in \eqref{eq:defp}, i.e., $p=5$. In this case, the geometric control hypothesis stated in \Cref{ass:linschroControllable} is not required because \eqref{eq:linschro} is STNC when $\omega$ is a nonempty open subset of $\Omega = (0,L)$. Extending \Cref{tm:mainresult1} and \Cref{tm:mainresult2} to the 3-dimensional case is also possible, albeit with additional technical difficulties, as described below. However, higher dimensional cases are more challenging due to the lack of an embedding $H^2(\Omega) \hookrightarrow L^{\infty}(\Omega)$.
\end{itemize}

\subsection{Extension to the 3-dimensional case}

Here, we outline possible extensions of our main results, \Cref{tm:mainresult1} and \Cref{tm:mainresult2}, to the 3-dimensional case, emphasizing the key differences. For $\Omega \subset \R^3$, a smooth bounded domain, the $L^2$-critical focusing nonlinearity in \eqref{eq:L2CriticalNLSRd} corresponds to the $p=4/3$ nonlinearity,
\begin{equation}
	\label{eq:L2CriticalNLS3d}
		\left\{
			\begin{array}{ll}
				 i \partial_t \psi + \Delta \psi = - |\psi|^{4/3} \psi,  & \text{ in }  (0,T) \times \Omega, 
				\\
				\psi = 0, & \text{ on } (0,T)\times \partial \Omega, 
				\\
				\psi(0, \cdot) = \psi_0, & \text{ in } \Omega.
			\end{array}
		\right.
\end{equation}
First, the local well-posedness of \eqref{eq:L2CriticalNLS3d} in $H^2(\Omega) \cap H_0^1(\Omega)$, as recalled in \Cref{lm:localwpH2H01}, trivially holds in dimension 3 thanks to the embedding $H^2(\Omega) \hookrightarrow L^{\infty}(\Omega)$. However, the local well-posedness in $H_0^1(\Omega)$, as described in \Cref{lm:localwpH01}, is not guaranteed a priori.

Without any additional assumptions, \Cref{thm:mainresultgod} still holds. The first part of \Cref{tm:globalwp} holds.

Now consider the nonlinear controlled equation:
\begin{equation}
	\label{eq:L2CriticalNLSControl3d}
		\left\{
			\begin{array}{ll}
				 i \partial_t \psi + \Delta \psi = - |\psi|^{4/3} \psi + v \mathds{1}_{\omega}, & \text{ in }  (0,T) \times \Omega, 
				\\
				\psi = 0, & \text{ on } (0,T)\times \partial \Omega, 
				\\
				\psi(0, \cdot) = \psi_0, & \text{ in } \Omega.
			\end{array}
		\right.
\end{equation}

The following generalization of \Cref{tm:mainresult1} holds:
\begin{tm}
\label{tm:mainresult13d}
Let $\Omega$ be a smooth bounded domain in $\R^3$. Let $n \geq 1$ and $x_1, \dots,x_n$ be $n$ distinct points in $\Omega$. Let $\omega \subset \Omega$ be an open set such that $x_1, \dots, x_n \in \omega$. Then, there exist $\delta>0$, $\lambda_0 >0$ such that for every $\lambda \geq \lambda_0$, there exists a a feedback control $v_{\lambda} \in L^{\infty}([0,\infty);H^2(\Omega) \cap H_0^1(\Omega))$, supported in $(0,T_{\lambda}) \times \omega$, and an associated global solution $\psi \in C([0,+\infty);H_0^1(\Omega))$ of \eqref{eq:L2CriticalNLSControl3d} starting from $\psi(0, \cdot) = \phi_{\lambda}(0,\cdot)+w_0$, where $\phi_\lambda$ is the blow-up profile in \Cref{thm:mainresultgod} and $\|w_0\|_{H^2\cap H^1_0(\Omega)}<\delta$. 
\end{tm}

\Cref{tm:mainresult13d} differs from \Cref{tm:mainresult1}. Indeed, the solution $\psi$ belongs to $C([0,+\infty);H_0^1(\Omega))$, meaning blow-up prevention is only achieved in the energy space $H_0^1(\Omega)$, rather than in $H^2(\Omega) \cap H_0^1(\Omega)$. This is due to the fact that Brezis-Gallouet's inequality does not hold in dimension 3, and therefore the second part of \Cref{tm:globalwp} does not hold a priori.\\

The following generalization of \Cref{tm:mainresult2} holds:
\begin{tm}
\label{tm:mainresult23d}
Let $\Omega$ be a smooth bounded domain in $\R^3$. Let $n \geq 1$ and $x_1, \dots,x_n$ be $n$ distinct points in $\Omega$. Let $\omega \subset \Omega$ be an open set such that $x_1, \dots, x_n \in \omega$, and assume \Cref{ass:linschroControllable} holds. Let $T>0$. Then, there exist $\delta>0$, $\lambda_0 >0$ such that for every $\lambda \geq \lambda_0$, there exists a control $v\in L^{\infty}([0,T];H^2(\Omega) \cap H_0^1(\Omega))$, supported in $(0,T) \times \omega$, and an associated solution $\psi \in C([0,T];H^2(\Omega) \cap H_0^1(\Omega))$ of \eqref{eq:L2CriticalNLSControl}, starting from $\psi(0, \cdot) = \phi_{\lambda}(0,\cdot)+w_0$, where $\phi_\lambda$ is the blow-up profile in \Cref{thm:mainresultgod} and $\|w_0\|_{H^2\cap H^1_0(\Omega)}<\delta$, satisfying 
\begin{equation}
\label{eq:psizero3d}
\psi(T,\cdot) = 0.
\end{equation}
\end{tm}

\Cref{tm:mainresult23d} generalizes \Cref{tm:mainresult2}. The controlled solution lies in the space $C([0,T];H^2(\Omega) \cap H_0^1(\Omega))$, where local well-posedness holds independently of $\Omega$. Additionally, small-time local null-controllability in $H^2(\Omega) \cap H_0^1(\Omega)$ is guaranteed under \Cref{ass:linschroControllable}. 
Proofs of these results are sketched in \Cref{sec:3dextension}, pointing out the main differences between the 2 and 3 dimensional cases. Aside of the well-posedness framework, the main difference between the proof of the main results in dimension 2 and 3 comes from the use of Kato's method in dimension 3. This method allows to weaken the distance on the complete space where the fixed-point argument is performed. In our case, the contraction estimates on the $p=4/3$ nonlinearity are only needed in $L^2(\Omega)$ (see \Cref{lm:cubicH2}). The adaptation of the proof of \Cref{tm:mainresult1} and \Cref{tm:mainresult2} to the 3-dimensional case being essentially only technical, we only provided sketches of the proofs to keep the presentation of the main ideas of the paper clear.

\subsection{Strategy of the proofs of the main results \Cref{tm:mainresult1} and  \Cref{tm:mainresult2}}

\textbf{Reduction to a single blow-up point.} To simplify the proof of \Cref{tm:mainresult1} and \Cref{tm:mainresult2}, we assume that there is only one blow-up point, denoted by $x_0 \in \Omega$. The case of a finite number of blow-up points is a straightforward adaptation of the following proof.\\

\textbf{Proof of \Cref{tm:mainresult1}.} We split the time interval $(0,+\infty)$ into three subintervals: $(0, t_1) \cup (t_1, t_2) \cup (t_2, +\infty)$, where $0 < t_1 < t_2 < T_{\lambda}$ will be appropriately chosen during the proof. We begin by treating the case $\psi(0,\cdot)=\phi_\lambda(0,\cdot)$. We then proceed to prove that, for the designed feedback law, there exists an open set around $\phi_\lambda(0,\cdot)$ such that the same conclusions to \Cref{tm:mainresult1} hold. \\

\textbf{\textit{Free evolution in $(0,t_1)$.}} In the first time interval $(0,t_1)$, we let the solution $\psi$ of \eqref{eq:L2CriticalNLSControl} evolve freely with $v=0$, starting from $\psi(0, \cdot) = \phi_{\lambda}(0,\cdot)$. Then, by the local well-posedness in $H^2(\Omega) \cap H_0^1(\Omega)$, $\psi = \phi_{\lambda}$ in $(0,t_1)$. A key step here is that we will choose $t_1$ so that the solution $\psi$ is highly concentrated near the blow-up point. Consequently, the solution is exponentially decaying away from the blow-up point due to \eqref{eq:estimationrlambda}, the time rescaling of the solution \eqref{eq:defRLambda}, and the exponential decay of $Q$ in \eqref{eq:expdecayQ} outside $0$. \\

\textbf{\textit{Rapid stabilization in $(t_1,t_2)$.}} In the second time interval $(t_1, t_2)$, we choose the control $v\in C([t_1,t_2];H^2(\Omega) \cap H_0^1(\Omega))$, supported in $\omega$ (a neighborhood containing the blow-up point), in a nonlinear feedback form appropriately designed to achieve rapid stabilization of the solution $\psi$ in $\omega$. This ensures exponential decay of the controlled solution in $\omega$. A fixed-point argument is employed to compensate for errors outside $\omega$, thanks to the smallness of the blow-up profile in the exterior region. This control strategy ensures that $\|\psi(t_2, \cdot)\|_{L^2(\Omega)} <  \|Q\|_{L^2(\R^2)}$. See \Cref{pr:rapidstab} below.\\

\textbf{\textit{Free evolution in $(t_2, +\infty)$.}} In the last time interval $(t_2, +\infty)$, we set $v = 0$. By using the global well-posedness in $H^2(\Omega) \cap H_0^1(\Omega)$ for data with $L^2$-norm smaller than $\|Q\|_{L^2(\mathbb{R}^2)}$, as recalled in \Cref{tm:globalwp}, the solution $\psi \in C([t_2, +\infty);H^2(\Omega) \cap H_0^1(\Omega))$ becomes global.\\

\textbf{\textit{Definition of a global solution.}} The combination of all these steps allows for the construction of a control $v\in L^{2}([0,+\infty);H^2(\Omega) \cap H_0^1(\Omega))$ such that the solution $\psi \in C([0, +\infty);H^2(\Omega) \cap H_0^1(\Omega))$ of \eqref{eq:L2CriticalNLSControl} is global. See \Cref{sec:endproofTh1} below.\\

\textbf{\textit{Extension to an open set around the blow-up profile.}} The proof above may be extended to an open set of initial data in $H^2(\Omega) \cap H_0^1(\Omega)$ around $\phi_\lambda(0,\cdot)$ by designing a nonlinear feedback law $v_\lambda=\mathcal{K}_\lambda(\psi)$ in the time interval $(0,t_1)$ such that $\mathcal{K}_\lambda(\phi_\lambda(t,\cdot))=0,\ t\in (0,t_1)$. Since $\mathcal{K}_\lambda(\phi_\lambda(t,\cdot))=0, \ t\in (0,t_1)$, the free evolution of the blow-up profile is maintained. Moreover, the nonlinear feedback law $\mathcal{K}_\lambda$ is designed to prevent the nonlinear term $|\psi|^2\psi$ to drive the solution away from the blow-up profile. Combined with a certain smallness of $\phi_\lambda(t,\cdot),\ t\in (0,t_1)$ outside of the control region, we perform a fixed-point argument allowing us to conclude that there exists an open set  in $H^2(\Omega) \cap H_0^1(\Omega)$ around $\phi_\lambda(0,\cdot)$ such that the solutions remains close to $\phi_\lambda(t_1,\cdot)$ at time $t_1$. We then prove that the nonlinear feedback control designed in the time interval $t\in (t_1,t_2)$ is sufficient to drive the solution in a state such that $\|\psi(t_2, \cdot)\|_{L^2(\Omega)} <  \|Q\|_{L^2(\R^2)}$, allowing us to draw the same conclusion as for the solution $\psi$ starting from $\phi_\lambda(0,\cdot)$.\\

\textbf{Proof of \Cref{tm:mainresult2}.} We split the time interval $(0,T)$ into three subintervals: $(0, t_1) \cup (t_1, t_2) \cup (t_2, T)$, where $0 < t_1 < t_2 < T_{\lambda} < T/2$ will be appropriately chosen during the proof. Notice that, by \Cref{thm:mainresultgod}, $T_\lambda$ may always be chosen such that $T_{\lambda} < T/2$ by taking $\lambda>0$ sufficiently large. Again, we begin by detailing the proof for the blow-up profile, before proving the existence of an open set of initial data close to it such that the conclusions of \Cref{tm:mainresult2} hold.\\

\textbf{\textit{Free evolution in $(0,t_1)$.}} In the first time interval $(0,t_1)$, we let the solution $\psi$ of \eqref{eq:L2CriticalNLSControl} evolve freely with $v=0$, starting from $\psi(0, \cdot) = \phi_{\lambda}(0,\cdot)$. Then, by the local well-posedness in $H^2(\Omega) \cap H_0^1(\Omega)$, $\psi = \phi_{\lambda}$ in $(0,t_1)$. Again, we will choose $t_1$ so that the solution $\psi$ is highly concentrated near the blow-up point. Consequently, the solution is exponentially decaying away from the blow-up point due to \eqref{eq:estimationrlambda}, the time rescaling of the solution \eqref{eq:defRLambda}, and the exponential decay of $Q$ in \eqref{eq:expdecayQ} outside $0$.\\

\textbf{\textit{Rapid stabilization in $(t_1,t_2)$.}} In the second time interval $(t_1, t_2)$, we choose the control $v\in L^{2}([t_1,t_2];H^2(\Omega) \cap H_0^1(\Omega))$, supported in $\omega$ (a neighborhood containing the blow-up point), in a nonlinear feedback form appropriately designed to achieve rapid stabilization of the solution $\psi$ in $\omega$. This ensures exponential decay of the controlled solution in $\omega$. A fixed-point argument is employed to compensate for errors outside $\omega$, thanks to the smallness of the blow-up profile in the exterior region. This control strategy ensures that $\|\psi(t_2, \cdot)\|_{H^2(\Omega)} <  \delta$, where $\delta$ is the radius of the ball in $H^2(\Omega)$ for which local null-controllability of \eqref{eq:L2CriticalNLSControl} holds in time $T/2$. See \Cref{pr:rapidstab} below.\\

\textbf{\textit{Local null-controllability in $(t_2, T)$.}} In the last time interval $(t_2, T)$, with $T-t_2 > T/2$, one can find a control $v\in L^{2}([t_2,T];H^2(\Omega) \cap H_0^1(\Omega))$, supported in $(t_2,T) \times \omega$, steering the solution $\psi$ of \eqref{eq:L2CriticalNLSControl} to $0$ in time $T$. See \Cref{thm:localnullcontrollabilityNLS} below and \Cref{sec:proofth2}.\\

\textbf{\textit{Definition of a null-controllable blow-up profile.}} The combination of all these steps allows for the construction of a control $v\in L^{\infty}([0,T];H^2(\Omega) \cap H_0^1(\Omega))$, supported in $(0,T)\times\omega$, such that the solution $\psi \in C([0, T];H^2(\Omega) \cap H_0^1(\Omega))$ of \eqref{eq:L2CriticalNLSControl} satisfies $\psi(T,\cdot)=0$.\\

\textbf{\textit{Extension to an open set around the null-controllable blow-up profile.}} We proceed as for the extension step of the proof of \Cref{tm:mainresult2}. First, we design a nonlinear feedback law $\mathcal{K}_\lambda$ in the time interval $t\in (0,t_1)$ such that $\mathcal{K}_\lambda(\phi_\lambda(t,\cdot))=0,\ t\in (0,t_1)$, and such that the nonlinear term may be mediated to ensure that solutions starting from an initial data near $\phi_\lambda(0,\cdot)$ remains close to the blow-up profile in $t\in (0,t_1)$. In the time interval $t\in (t_1,t_2)$, we prove that the nonlinear feedback law $\mathcal{K}$ allows solution $\psi$ starting near from $\phi_\lambda(t_1,\cdot)$ satisfy $\|\psi(t_2, \cdot)\|_{H^2(\Omega)} <  \delta$, where $\delta$ is the radius of the ball in $H^2(\Omega)$ for which local null-controllability of \eqref{eq:L2CriticalNLSControl} holds in time $T/2$. Applying the small-time  null-controllability of \eqref{eq:L2CriticalNLSControl}, the conclusions of \Cref{tm:mainresult2} hold.

\section{Proof of the main results}

We begin this section by establishing useful technical lemmas. 

\subsection{Useful lemmas}

Let $\omega_1=B(x_0,r) \subset \omega$ with $r>0$ be such that $B(x_0,2r)\subset \omega$. The first lemma states that outside $\omega_1$, the function $R_{\lambda}$ defined in \eqref{eq:defRLambda} decays exponentially.
\begin{lm}\label{lem:R}
There exists $C, \delta=\delta(r), \lambda_0>0$ such that for any $\lambda \geq \lambda_0$,
\begin{equation}
\label{eq:boundRLambdaOutsideomega}
\|R_\lambda(t,\cdot) \|_{H^2(\Omega \setminus \omega_1)} \leq C e^{-\frac{\delta}{\lambda(T_\lambda-t)}}, \quad \forall t\in [0, T_{\lambda}).
\end{equation}
\end{lm}
\begin{proof}
For $x \in \Omega \setminus \omega_1$, note that $r < |x-x_0|$.
Then from the definition of $R_{\lambda}$ given in \eqref{eq:defRLambda} and the estimates on $Q$ in \eqref{eq:expdecayQ}, we immediately obtain \eqref{eq:boundRLambdaOutsideomega} for some $C>0$ and $\delta = \delta(r)>0$.
\end{proof}

The second lemma consists in establishing estimates on the blow-up profile defined in \eqref{eq:decompositionphilambda}.
\begin{lm}
\label{lem:philambda}
There exists $C, \alpha, \beta, \lambda_0 >0$ such that for any $\lambda \geq \lambda_0$, 
\begin{equation}
\label{eq:philambdatH2}
\|\phi_{\lambda}(t, \cdot)\|_{H^2(\Omega)} \leq  \dfrac{C}{\lambda^{\alpha}(T_\lambda-t)^\beta}\qquad \forall t \in [0,T_{\lambda}).
\end{equation}
\end{lm}
\begin{proof}
This is a straightforward application of the decomposition of $\phi_\lambda$ as in \eqref{eq:decompositionphilambda}, the definition of $R_\lambda$ in \eqref{eq:defRLambda}, the estimates on $Q$ in \eqref{eq:expdecayQ} and the estimate on $r_\lambda$ in \eqref{eq:estimationrlambda}.
\end{proof}

\subsection{Rapid stabilization of the blow-up profile in the control zone}

In this part, we aim at establishing the following approximate controllability result for data close to $\phi_{\lambda}(0,\cdot)$. 

\begin{pr}
\label{pr:rapidstab}
For every $\varepsilon >0$, there exist $\delta>0$, $\lambda_0 >0$ such that for every $\lambda \geq \lambda_0$, there exist  $t_1 < t_2 \in (0, T_{\lambda})$ and a feedback control $v = v(t,\cdot) = \mathcal K(\psi(t,\cdot)) \in L^\infty((0,t_2);H^2(\Omega) \cap H_0^1(\Omega))$ such that the system
\begin{equation}
	\label{eq:L2CriticalNLSt2}
		\left\{
			\begin{array}{ll}
				 i \partial_t \psi + \Delta \psi = - |\psi|^{2} \psi + v 1_{\omega}  & \text{ in }  (0,t_2) \times \Omega, 
				\\
				\psi = 0 & \text{ on } (0,t_2)\times \partial \Omega, 
				\\
				\psi(0, \cdot) = \phi_{\lambda}(0, \cdot)+w_0(\cdot) & \text{ in } \Omega,
			\end{array}
		\right.
\end{equation}
with $\|w_0\|_{H^2\cap H^1_0(\Omega)} < \delta$, has a unique solution $\psi \in C([0,t_2];H^2(\Omega) \cap H_0^1(\Omega))$ satisfying
\begin{equation}
\label{eq:psipetitt2}
\|\psi(t_2,\cdot)\|_{H^2\cap H^1_0(\Omega)} \leq \varepsilon.
\end{equation}
\end{pr}
\begin{rmk}
From the proof of \Cref{pr:rapidstab}, the parameter $\delta>0$ can be chosen of the following form $\delta = \frac{ \varepsilon}{16}$.
\end{rmk}
\begin{rmk}
Notice that in dimension $d=2$, the non-linearity $|\psi|^2 \psi$ is algebraic, whereas in dimension $d=3$, the non-linearity $|\psi|^{4/3} \psi$ is not. This will simplify the proof in the two-dimensional case.
\end{rmk}

The proof of \Cref{pr:rapidstab} consists in choosing the control $v$ to get rapid stabilization of the function $\psi$ near the blow-up point $x_0$.

\begin{proof}
Let $t_1 < t_2 \in (0, T_{\lambda})$ to be fixed later. We split the proof in several steps.\\

\noindent \textbf{Step 1: Stabilization around the blow-up profile.} We design a first nonlinear feedback law $v_\lambda=\mathcal{K}_\lambda^1(\psi)$ such that the solution starting near the blow-up profile $\phi_\lambda$ at time $t=0$ remains near $\phi_\lambda$ at time $t=t_1$. For $t\in (0,t_1)$, let,
\begin{equation}\label{eq:defK1}
v_\lambda(t,x)=\mathcal{K}_\lambda^1(\psi(t,x))=\chi(x)( |\psi(t,x)|^2\psi(t,x) - | \phi_\lambda(t,x)|^2 \phi_\lambda(t,x)),
\end{equation}
where $\chi$ a smooth cut-off function belonging to $C^{\infty}(\Omega;[0,1])$ such that
\begin{equation}
\label{eq:defchiBis}
\chi = 
		\left\{
			\begin{array}{ll}
				  1  & \text{ in }  \omega_1, 
				\\
				0 & \text{ in } \Omega \setminus \omega.
			\end{array}
		\right.
\end{equation}
Notice that $\mathcal{K}_1(\phi_\lambda(t,x))=0$, and therefore the solution $\psi$ of \eqref{eq:L2CriticalNLSt2} starting from $\phi_\lambda(0,\cdot)$ with $v$ defined by \eqref{eq:defK1} satisfy $\psi \equiv \phi_\lambda$. To prove that solutions starting from an initial data close to $\phi_\lambda(0,\cdot)$ remain close to $\phi_\lambda$ along time, we set up a fixed-point argument by letting $\psi(t,x)=\phi_\lambda(t,x) + w(t,x)$, with $w(0,x)=w_0(x)$, where 
\begin{align*}
(i\partial_t + \Delta) w & =|\phi_\lambda|^2\phi_\lambda - |\psi|^2\psi + \mathds{1}_\omega v \\
& = (1-\chi(x))( |\phi_\lambda|^2\phi_\lambda - |\psi|^2\psi ).
\end{align*}
Let $(S(t))_{t \in \R}$ be the unitary group associated to $i\Delta$ over $H=H^2(\Omega) \cap H^1_0(\Omega)$.
First, we consider the following fixed-point argument, 
\[
\Gamma_1(w(t,\cdot))=S(t)w_0 + \int_0^t S(t-s) (1-\chi(x))( |\phi_\lambda|^2\phi_\lambda - |\psi|^2\psi )(s,\cdot)ds 
\]
in the Banach space  
\begin{equation}
\label{eq:spacefixedpoint1}
X_{M_1}:= \left\{ w\in C([0,t_1];H) \, \left| \, \|w\|_{X} \leq {M_1} \right. \right\}, \quad \textrm{where }  \|w\|_{X}:=\|w\|_{C([0,t_1];H)}.
\end{equation}

We first prove the energy estimate. Notice that $\chi$ is supported outside $\omega_1$. Therefore, there exists $C>0$ such that,
\begin{equation}
\label{eq:outsideomega1H2}
\|(1-\chi) f \|_{H^2(\Omega)}  \leq C \|f\|_{H^2(\Omega \setminus \omega_1)}\qquad \forall f \in H^2(\Omega).
\end{equation}
By using the decomposition of $\phi_{\lambda}$ in \eqref{eq:decompositionphilambda}, the estimate on $r_{\lambda}$ in \eqref{eq:estimationrlambda}, the estimate \eqref{eq:boundRLambdaOutsideomega} of $R_{\lambda}$ in Lemma \ref{lem:R}, the estimate \eqref{eq:outsideomega1H2}, we have the following sequence of estimates
\begin{align}
\left\| \Gamma_1(w(t,\cdot)) \right\|_X & \leq \delta + \left\|  \int_0^t S(t-s) (1-\chi(x))( |\phi_\lambda|^2\phi_\lambda - |\psi|^2\psi )(s,\cdot)ds   \right\|_X \notag\\ 
& \leq \delta +   \int_0^{t_1} \left\| (1-\chi(x))( |\phi_\lambda|^2\phi_\lambda - |\psi|^2\psi )(s,\cdot) \right\|_{H^2(\Omega)} ds   \notag \\
& \leq \delta +  \int_0^{t_1} \left\|  (|\phi_\lambda|^2\phi_\lambda - |\psi|^2\psi )(s,.) \right\|_{H^2(\Omega \setminus \omega_1)} ds \notag  \\
& \leq \delta +  C_1 \int_0^{t_1}  \left(\|\phi_\lambda(s,\cdot)\|_{H^2(\Omega \setminus \omega_1)}^2+ \|\psi(s,\cdot)\|_{H^2(\Omega \setminus \omega_1)}^2\right)\| w(s,\cdot)\|_{H^2(\Omega \setminus \omega_1)} ds \notag  \\
& \leq \delta +  C_1 \int_0^{t_1}  \left(e^{-\frac{\delta'}{\lambda(T_\lambda-t)}}+ \|w(s,.)\|_{H^2(\Omega \setminus \omega_1)}^2\right)\| w(s,\cdot)\|_{H^2(\Omega \setminus \omega_1)} ds \notag\\
& \leq \delta +  C_1t_1  \left(e^{-\frac{\delta'}{\lambda T_\lambda}}+ \|w(s,\cdot)\|_{X}^2\right)\| w(s,\cdot)\|_{X}\notag \\
& \leq \delta +  C_1t_1  \left(e^{-\frac{\delta'}{\lambda T_\lambda}}+M_1^2\right){M_1},\label{eq:estimateballfirstfixed}
\end{align}
with $C_1>0$ changing from line to line and $\delta'(r,\kappa)=\delta'>0$. On the other hand, denote $\psi_w =\phi_\lambda+w$ and $\psi_z = \phi_\lambda+z$ associated to $w,z \in X_{M_1}$. We have, since the term $|\phi_\lambda|^2\phi_\lambda$ cancels out,
\begin{align}
&\left\| \Gamma_1(w(t,\cdot)) \right.  \left. - \Gamma_1(z(t,\cdot)) \right\|_X\notag \\
& \leq  \left\|  \int_0^t S(t-s) (1-\chi(x))( |\psi_w|^2\psi_w  - |\psi_z|^2\psi_z )(s,\cdot)ds   \right\|_X\notag \\ 
& \leq  C_2  \int_0^{t_1} \left\| (1-\chi(x))( |\psi_w|^2\psi_w  - |\psi_z|^2\psi_z)(s,\cdot) \right\|_{H^2(\Omega)} ds \notag   \\
& \leq C_2  \int_0^{t_1} \left\|  (|\psi_w|^2\psi_w  - |\psi_z|^2\psi_z )(s,\cdot) \right\|_{H^2(\Omega \setminus \omega_1)} ds \notag  \\
& \leq C_2 \int_0^{t_1}  \left(\|\psi_w(s,\cdot)\|_{H^2(\Omega \setminus \omega_1)}^2+ \|\psi_z (s,.)\|_{H^2(\Omega \setminus \omega_1)}^2\right)\| (\psi_w - \psi_z)(s,\cdot)\|_{H^2(\Omega \setminus \omega_1)} ds \notag  \\
& \leq C_2 \int_0^{t_1}  \left(\|(\phi_\lambda + w)(s,\cdot)\|_{H^2(\Omega \setminus \omega_1)}^2+ \|(\phi_\lambda+z) (s,\cdot)\|_{H^2(\Omega \setminus \omega_1)}^2\right)\| (w - z)(s,\cdot)\|_{H^2(\Omega \setminus \omega_1)} ds \notag  \\
& \leq C_2 \int_0^{t_1}  \left(2\|\phi_\lambda(s,\cdot)\|_{H^2(\Omega \setminus \omega_1)}^2 +  \|w(s,\cdot)\|_{H^2(\Omega \setminus \omega_1)}^2+ \|z (s,\cdot)\|_{H^2(\Omega \setminus \omega_1)}^2\right)\| (w - z)(s,\cdot)\|_{H^2(\Omega \setminus \omega_1)} ds \notag  \\
& \leq  C_2 t_1  \left(2e^{-\frac{\delta'}{\lambda T_\lambda }}+ \|w \|_{X}^2 + \|z \|_{X}^2 \right)\| w - z \|_{X} \notag\\
& \leq  C_2 t_1  \left(e^{-\frac{\delta'}{\lambda T_\lambda }}+ M_1^2 \right)\| w - z \|_{X}, \label{eq:estimatecontractionfirstfixed}
\end{align}
with $C_2>0$ changing from line to line. From \Cref{thm:mainresultgod}, we have that there exists $c>0$, independent of $\lambda$, such that
\begin{equation}\label{eq:Tlambda}
T_{\lambda} \in (0, c \lambda^{-2}).
\end{equation}
We can therefore consider $\lambda>0$ sufficiently large so that $T_\lambda < 1/4 $. Let us fix
\begin{equation}
\label{eq:deft1}
t_1(\lambda)=t_1:=T_\lambda (1-2T_\lambda). 
\end{equation}

\begin{cl} 
\label{cl:firstfixedpoint}
There exists $\lambda_0>0$ sufficiently large such that for 
\begin{equation}\label{eq:defM1}
M_1 \in \left(0, \left( \dfrac{1 }{4Cc\lambda_0^{-2}(1-2c\lambda_0^{-2})}\right)^{1/2}\right],
\end{equation}
with $C=\max(C_1,C_2)$, we have
\begin{equation}
\label{eq:estimationsfirstfixedpoint}
t_1  \left(e^{-\frac{\delta'}{\lambda T_\lambda }}+ M_1^2 \right)< \dfrac{1}{2C}.
\end{equation}
\end{cl}
\begin{proof}[Proof of \Cref{cl:firstfixedpoint}]
Indeed, we have, using $\lambda_0<\lambda$ and $T_\lambda<1/2$,
\begin{align*}
t_1  e^{-\frac{\delta'}{\lambda T_\lambda }} = T_\lambda (1-2T_\lambda)  e^{-\frac{\delta'}{\lambda T_\lambda }} < T_\lambda e^{-\frac{\delta'}{\lambda T_\lambda }} < \dfrac{ \lambda T_{\lambda}}{\lambda_0} e^{-\frac{\delta'}{\lambda T_{\lambda} }}.
\end{align*}
Since $x \in (0,\infty) \mapsto xe^{-\delta'/x}$ is increasing, and since $\lambda T_\lambda < c\lambda^{-1} < c\lambda_0^{-1}$ from the bound \eqref{eq:Tlambda}, we deduce for $\lambda_0>0$ large enough,
\begin{align*}
t_1  e^{-\frac{\delta'}{\lambda T_\lambda }} < \dfrac{ \lambda T_{\lambda}}{\lambda_0} e^{-\frac{\delta'}{\lambda T_{\lambda} }} <\dfrac{c}{\lambda_0^2} e^{-\frac{ \lambda_0 \delta'}{c }} < \dfrac{1}{4C}.
\end{align*}
Moreover, from the definition of $t_1$ in \eqref{eq:deft1}, since $x\in (0,1/4) \mapsto x(1-2x)$ and $T_\lambda<c\lambda^{-2}<c\lambda_0^{-2}<1/4$ for $\lambda_0$ sufficiently large, we have
\[
t_1<c\lambda_0^{-2}(1-2c\lambda_0^{-2}),
\]
and thus the choice \eqref{eq:defM1} leads to \eqref{eq:estimationsfirstfixedpoint}.
\end{proof}

By defining 
\begin{equation}\label{def:delta}
\delta:=M_1/2,
\end{equation}
and using \eqref{eq:estimateballfirstfixed}, \eqref{eq:estimatecontractionfirstfixed} and \eqref{eq:estimationsfirstfixedpoint}, we are able to see that $\Gamma_1$ maps $X_{M_1}$ to itself, and that $\Gamma_1$ is a contraction over $X_{M_1}$. Therefore by Banach's fixed point theorem, $\Gamma_1
$ admits a fixed-point in $X_{M_1}$. At the end of this step, $M_1$ is still not fixed but chosen in the interval of \eqref{eq:defM1} to guarantee that $\Gamma_1$ has a unique fixed-point. \\

\noindent \textbf{Step 2: Ansatz and feedback control.} We look for a solution $\psi$ to \eqref{eq:L2CriticalNLSt2} on $(t_1, t_2)$, with a control $v$ to be determined later, of the form
\begin{equation}
\label{eq:ansatz}
\psi(t,x)=e^{-\mu (t-t_1)}\phi_\lambda(t,x) + \tilde{w}(t,x)\qquad (t, x) \in (t_1, t_2) \times \Omega,
\end{equation}
where $\mu=\mu(\lambda)>2$ is a parameter to be determined later. Since $\phi_\lambda$ belongs to $C([0,T_\lambda);H^2(\Omega)\cap H^1_0(\Omega))$, so does $e^{-\mu (t-t_1)}\phi_\lambda$. Moreover, since $e^{-\mu (t-t_1)}\phi_\lambda |_{t=t_1}=\phi_\lambda(t_1,\cdot)$, we look for $\tilde{w} \in C([t_1,t_2];H^2(\Omega)\cap  H^1_0(\Omega))$ such that the continuity with the solution from step 1 is ensured, that is, 
\[
\tilde{w}(t_1,\cdot)=w(t_1,\cdot), \text{ in } H^2(\Omega)\cap  H^1_0(\Omega)
\]
where $w$ was obtained from a fixed-point argument on $\Gamma_1$. Therefore, 
\[
\|w(t_1,\cdot)\|_{H^2\cap H^1_0(\Omega)} = \|\tilde{w}(t_1,\cdot)\|_{H^2\cap H^1_0(\Omega)} \leq M_1. 
\]
Since $\phi_\lambda$ satisfies \eqref{eq:L2CriticalNLS}, we obtain
\begin{align*}
(i\partial_t + \Delta)[e^{-\mu (t-t_1)} \phi_\lambda] & =-i\mu e^{-\mu (t-t_1)} \phi_\lambda +  e^{-\mu (t-t_1)} (i \partial_t + \Delta) \phi_\lambda \\
& =-i\mu e^{-\mu (t-t_1)} \phi_\lambda -  e^{-\mu (t-t_1)} | \phi_\lambda |^{2} \phi_\lambda\qquad \text{in}\ (t_1, t_2) \times \Omega.
\end{align*}
Therefore, we set the second nonlinear feedback law $\mathcal{K}_\lambda^2$ of the state $\psi$
\begin{equation}
\label{eq:defcontrolfeedback}
v(t,x)=\mathcal{K}_\lambda^2(\psi(t,x))=\chi(x)\left( |\psi|^{2}\psi - e^{-\mu (t-t_1)} \left( |\phi_\lambda|^{2} \phi_\lambda + i \mu \phi_\lambda \right)\right) \qquad \text{in}\ (t_1, t_2) \times \Omega,
\end{equation}
where $\chi$ is smooth cut-off function defined in \eqref{eq:defchi}. Injecting the ansatz \eqref{eq:ansatz} in \eqref{eq:L2CriticalNLSControl} with the nonlinear feedback control \eqref{eq:defcontrolfeedback} gives, 
\begin{align}
\label{eq:w}
(i\partial_t+\Delta )\tilde{w} =& -(1-\chi) \left( |\psi|^{2}\psi - e^{-\mu (t-t_1)} \left( |\phi_\lambda|^{2} \phi_\lambda + i \mu \phi_\lambda \right)\right) \nonumber \\
 =& -(1-\chi) \left( |\psi|^{2}\psi - e^{-3\mu(t-t_1)} |\phi_\lambda|^{2} \phi_\lambda   \right.  \nonumber \\
 & \qquad    \left.  + e^{-3\mu(t-t_1)} |\phi_\lambda|^{2} \phi_\lambda - e^{-\mu (t-t_1)} \left( |\phi_\lambda|^{2} \phi_\lambda + i \mu \phi_\lambda \right)\right), \quad \text{in}\ (t_1, t_2) \times \Omega.
\end{align}
Defining,
\begin{align}
\label{eq:defN1}
N_1(\phi_\lambda) & := -(1-\chi) \left( e^{-3\mu(t-t_1)} |\phi_\lambda|^{2} \phi_\lambda  - e^{-\mu (t-t_1)} \left( |\phi_\lambda|^{2} \phi_\lambda + i \mu \phi_\lambda \right)\right),    \\
\label{eq:defN2} N_2(\tilde{w},\phi_\lambda) & :=-(1-\chi) \left( |\psi|^{2}\psi - e^{-3\mu(t-t_1)} |\phi_\lambda|^{2} \phi_\lambda \right),
\end{align}
we then obtain that $\tilde{w}$ satisfies,
\begin{equation}
\label{eq:wN1N2}
(i\partial_t+\Delta )\tilde{w} = N_1(\phi_\lambda) + N_2(\tilde{w},\phi_\lambda), \quad \text{in}\ (t_1, t_2) \times \Omega.
\end{equation}

\textbf{Step 3: A second fixed-point argument.} The goal of this part consists in proving that \eqref{eq:wN1N2} admits a unique solution $\tilde{w}\in C([t_1,t_2];H^2(\Omega) \cap H^1_0 (\Omega))$. Let $(S(t))_{t \in \R}$ be the unitary group associated to $i\Delta$ over $H=H^2(\Omega) \cap H^1_0(\Omega)$. We look for $\tilde{w}$ of the form,  
\begin{equation}
\label{eq:seekw}
\tilde{w}(t, \cdot)=S(t-t_1)w(t_1,\cdot) + \int_{t_1}^{t} S(t-\tau)(N_1(\phi_\lambda)(\tau)+N_2(\tilde{w},\phi_\lambda)(\tau)) d\tau, \quad t\in  [t_1,t_2].
\end{equation}
To prove the existence and uniqueness of $\tilde{w}$, we proceed with a fixed-point argument. Notice that in dimension $d=2$, the non-linearity $|\psi|^2 \psi$ is algebraic, whereas in dimension $d=3$, the non-linearity $|\psi|^{4/3} \psi$ is not. This will simplify the proof in the two-dimensional case. We aim at proving that
\begin{equation}
\label{eq:defGamma}
\Gamma_2(\tilde{w}(t, \cdot))=S(t-t_1)w(t_1,\cdot) +\int_{t_1}^{t} S(t-\tau)(N_1(\phi_\lambda)(\tau)+N_2(\tilde{w},\phi_\lambda)(\tau)) d\tau, \quad t\in  [t_1,t_2].
\end{equation}
has an unique fixed-point in the Banach space 
\begin{equation}
\label{eq:spacefixedpoint}
X_{M_2}:= \left\{ \tilde{w} \in C([t_1,t_2];H) \, \left| \, \|\tilde{w}\|_{X} \leq M_2 \right. \right\}, \quad \textrm{where }  \|\tilde{w}\|_{X}:=\|\tilde{w}\|_{C([t_1,t_2];H)}.
\end{equation}
Notice that compared to \cite{BGT03, God11}, we do not need a exponential weight in time in the definition of $\|\cdot\|_X$, since we do not need $\tilde{w}(t,\cdot) \rightarrow 0$ as $t \rightarrow T_\lambda^{-}$, where $T_\lambda$ is the blow-up time. 

We can now proceed to estimate $N_1$ and $N_2$. First, we estimate $N_1$ defined in \eqref{eq:defN1}. By using the decomposition of $\phi_{\lambda}$ in \eqref{eq:decompositionphilambda}, the estimate on $r_{\lambda}$ in \eqref{eq:estimationrlambda}, the estimate \eqref{eq:boundRLambdaOutsideomega} of $R_{\lambda}$ in Lemma \ref{lem:R}, the estimate \eqref{eq:outsideomega1H2}, we have the following sequence of estimates
\begin{align}
\left\| \int_{t_1}^{t} S(t-\tau)N_1(\phi_\lambda)(\tau) d\tau   \right\|_{X} \leq &  \int_{t_1}^{t_2} \| N_1(\phi_\lambda)(\tau) \|_{H^2(\Omega)} d\tau \notag\\
\leq &  \hat{C}_1 \int_{t_1}^{t_2} \| N_1(\phi_\lambda)(\tau) \|_{H^2(\Omega \setminus \omega_1)} d\tau  \notag \\
\leq & \hat{C}_1 (1+\mu) \int_{t_1}^{t_2} \|\phi_\lambda(\tau)\|_{H^2(\Omega \setminus \omega_1)}+ \|\phi_\lambda(\tau)\|_{H^2(\Omega \setminus \omega_1)}^{3} d\tau \notag \\
\leq  & \hat{C}_1 (1+\mu) \int_{t_1}^{t_2} e^{- \frac{\delta'}{\lambda(T_\lambda - \tau)}} d \tau  \notag\\ 
\leq  & \hat{C}_1 (1+\mu)(t_2-t_1) e^{- \frac{\delta'}{\lambda(T_\lambda - t_1)}},  \qquad \forall t \in (t_1,t_2),
\label{eq:estimateN1}
\end{align}
for some positive constant $\hat{C}_1>0$ changing from line to line and $\delta'=\delta'(r, \kappa)>0$.

To estimate $N_2$ defined in \eqref{eq:defN2}, we write
\begin{align*}
|\psi|^2 \psi - e^{-3 \mu (t-t_1)} |\phi_{\lambda}|^2 \phi_{\lambda} =& |\tilde{w}|^2 \tilde{w} + 2 e^{-\mu (t-t_1)} |\tilde{w}|^2 \phi_{\lambda} + e^{-2 \mu(t-t_1) }\overline{\tilde{w}} \phi_{\lambda}^2 \\
&+ 2 e^{-2\mu (t-t_1)} |\phi_{\lambda}|^2 \tilde{w} + e^{-\mu (t-t_1)} \overline{\phi_{\lambda}} \tilde{w}^2.
\end{align*}
This gives, by Young's inequality and by using that $H^2(\Omega)$ is an algebra,
\begin{align}
&\left\| \int_{t_1}^{t} S(t-\tau)N_2(\tilde{w},\phi_\lambda)(\tau) d\tau   \right\|_{X} \notag  \\
& \leq  \hat{C}_2 \int_{t_1}^{t_2} \| N_2(\tilde{w},\phi_\lambda)(\tau) \|_{H^2(\Omega \setminus \omega_1)}  d\tau \notag\\
& \leq \hat{C}_2 \int_{t_1}^{t_2} \left( \| \tilde{w}(\tau, \cdot) \|_{H^2(\Omega \setminus \omega_1)}^2 +  \| \phi_\lambda(\tau,\cdot) \|_{H^2(\Omega \setminus \omega_1)}^2 \right) \| \tilde{w}(\tau, \cdot) \|_{H^2(\Omega \setminus \omega_1)}   d\tau \notag \\
 & \leq \hat{C}_2(t_2-t_1) (M_2^2 + e^{- \frac{ \delta'}{\lambda(T_\lambda - t_1)}} ) M_2 \qquad \forall t \in (t_1,t_2). \label{eq:estimateN2}
\end{align}
for some positive constant $\hat{C}_2>0$ changing from line to line. By \eqref{eq:defGamma}, \eqref{eq:estimateN1} and \eqref{eq:estimateN2}, we have that for any $\tilde{w} \in X_{M_2}$,
\begin{align}\label{eq:GammaSameBall}
\| \Gamma_2(\tilde{w})\|_X \leq M_1 + \tilde{C}_1(t_2-t_1) \left( (1+\mu)e^{- \frac{\delta'}{\lambda(T_\lambda - t_1)}} +   (M_2^2 + e^{- \frac{ \delta'}{\lambda(T_\lambda - t_1)}} ) M_2 \right),
\end{align}
with $\tilde{C}_1=\max\{\hat{C}_1,\hat{C}_2\}$.

The contraction estimates are done the following way. Denote $\psi_{\tilde{w}} = \phi_\lambda+\tilde{w}$ and $\psi_{\tilde{z}} = \phi_\lambda+\tilde{z}$ associated to $\tilde{w},\tilde{z}\in X_{M_2}$ respectively. For any  $\tilde{w},\tilde{z}\in X_{M_2}$, we have by the Sobolev embedding $H^2(\Omega \setminus \omega_1)\hookrightarrow L^\infty(\Omega \setminus \omega_1)$,
\begin{align}
\|\Gamma_2({\tilde{w}} ) - \Gamma_2({\tilde{z}} )\|_{X} & \leq \int_{t_1}^{t_2} \| (1-\chi)(|\psi_{\tilde{w}} |^{2}\psi_{\tilde{w}}  - |\psi_{\tilde{z}} |^{2}\psi_{\tilde{z}} ) \|_{H^2(\Omega \setminus \omega_1)} d \tau \notag \\
& \leq \tilde{C}_2 \int_{t_1}^{t_2} \| |\psi_{\tilde{w}} |^{2}\psi_{\tilde{w}}  - |\psi_{\tilde{z}} |^{2}\psi_{\tilde{z}}  \|_{H^2(\Omega \setminus \omega_1)} d \tau \notag\\
& \leq \tilde{C}_2 \int_{t_1}^{t_2} \| |\phi_\lambda+\tilde{w}|^{2}(\phi_\lambda+\tilde{w}) - |\phi_\lambda+\tilde{z} |^{2}(\phi_\lambda+\tilde{z})  \|_{H^2(\Omega \setminus \omega_1)} d \tau\notag \\
& \leq \tilde{C}_2 \int_{t_1}^{t_2} \left( \| \phi_\lambda + {\tilde{w}}  \|_{H^2(\Omega \setminus \omega_1)}^{2} +  \| \phi_\lambda + {\tilde{z}} \|_{H^2(\Omega \setminus \omega_1)}^{2} \right) \|{\tilde{w}}  - {\tilde{z}} \|_{H^2(\Omega \setminus \omega_1)} d \tau\notag \\
& \leq \tilde{C}_2 \int_{t_1}^{t_2} \left( \| \phi_\lambda \|_{H^2(\Omega \setminus \omega_1)}^{2} + \|{\tilde{w}}  \|_{H^2(\Omega \setminus \omega_1)}^{2} + \|{\tilde{z}} \|_{H^2(\Omega \setminus \omega_1)}^{2} \right) \|{\tilde{w}}  - {\tilde{z}} \|_{H^2(\Omega \setminus \omega_1)} d \tau\notag \\
\|\Gamma_2({\tilde{w}} ) - \Gamma_2({\tilde{z}} )\|_{X} & \leq \tilde{C}_2(t_2-t_1) \left( e^{- \frac{ \delta'}{\lambda(T_\lambda - t_1)}}+ 2M_2^{2} \right) \|{\tilde{w}} -{\tilde{z}} \|_{X},
\end{align}
with $\tilde{C}_2>0$ a constant changing from line to line. 

Let us now fix
\begin{equation}
\label{eq:deft1t2}
t_2 = T_{\lambda}(1-T_{\lambda})
\end{equation}
and
\begin{equation}
\label{eq:defmut1t2}
\mu(t_2-t_1) = \frac{1}{\lambda(T_{\lambda}-t_2)}=\frac{1}{\lambda T_{\lambda}^2}.
\end{equation}
\begin{cl}
\label{cl:secondfixedpoint}
Let 
\begin{equation}
\label{eq:defM}
M_2 = \left(\frac{\lambda_0^4}{4\tilde{C}c^2}\right)^{\frac{1}{2}},
\end{equation}
and 
\begin{equation}
\label{eq:defM1Bis}
M_1 = \min\left( \left( \dfrac{1 }{4Cc\lambda_0^{-2}(1-2c\lambda_0^{-2})}\right)^{1/2} , \frac{M_2}{4}\right),
\end{equation}
where $\tilde{C}=\max\{\tilde{C}_1,\tilde{C}_2\}$ and $c>0$ is the constant appearing in \eqref{eq:Tlambda}. Then we have
\begin{equation}\label{bound:gamma2}
\|\Gamma_2(\tilde{w})\|_X \leq M_2,\qquad \forall \tilde{w}\in X_{M_2},
\end{equation}
and
\begin{equation}
\label{eq:contractionprop}
\|\Gamma_2({\tilde{w}} ) - \Gamma_2({\tilde{z}} )\|_{X} \leq  \frac{3}{4} \|{\tilde{w}} -{\tilde{z}} \|_{X}\qquad \forall \tilde{w},\ \tilde{z} \in X_{M_2}.
\end{equation}
Furthermore, $M_1$ satisfies \eqref{eq:defM1}, allowing to ensure that both $\Gamma_1$ and $\Gamma_2$ admit a fixed point in $X_{M_1}$ and $X_{M_2}$ respectively.  
\end{cl}
\begin{proof}[Proof of \Cref{cl:secondfixedpoint}]
Let us first prove that, with these choices, the estimate \eqref{eq:GammaSameBall} implies indeed \eqref{bound:gamma2}. First, by the definition of $M_1$ given by \eqref{eq:defM1Bis},
\begin{equation}\label{eq:M1pluspetitM2}
M_1 \leq \dfrac{M_2}{4}.
\end{equation}
Moreover, using \eqref{eq:Tlambda}, \eqref{eq:deft1}, \eqref{eq:deft1t2},  \eqref{eq:defmut1t2} and, since $T_\lambda^2<\lambda T_\lambda^2 < c^2\lambda^{-3}<c^2\lambda_0^{-3}<1$ for $1<\lambda_0<\lambda$ with $\lambda_0$ sufficiently large, 
\begin{multline*}
\tilde{C}_1 (t_2-t_1)(1+\mu)e^{- \frac{\delta'}{\lambda(T_\lambda - t_1)}} 
  =  \tilde{C} T_\lambda^2 \left(1+\dfrac{1}{\lambda T_\lambda^2}\right)e^{- \frac{\delta'}{2\lambda T_\lambda^2}} 
 \\  \leq  \tilde{C} \left(1+\dfrac{1}{\lambda T_\lambda^2}\right)e^{- \frac{\delta'}{2\lambda T_\lambda^2}}
  \leq  \dfrac{2\tilde{C} }{\lambda T_\lambda^2}e^{- \frac{\delta'}{2\lambda T_\lambda^2}}
  =  \dfrac{4\tilde{C} }{2\lambda T_\lambda^2}e^{- \frac{\delta'}{2\lambda T_\lambda^2}} 
 \leq \frac{4\tilde{C}\lambda_0^3}{2c^2} e^{- \frac{\lambda_0^3 \delta'}{2c^2}},
\end{multline*}
where we have used the fact that $x\in (0,m) \mapsto e^{-\delta'/2x}/2x$ is increasing for $m>0$ small enough and $\lambda T_\lambda^2< c^2/\lambda^3 <  c^2/\lambda_0^3<m$ for $\lambda_0$ sufficiently large. Thus, using $\lambda_0$ sufficiently large, 
\begin{equation}\label{estim:g21}
\tilde{C}_1 (t_2-t_1)(1+\mu)e^{- \frac{\delta'}{\lambda(T_\lambda - t_1)}} \leq \frac{4\tilde{C}\lambda_0^3}{2c^2} e^{- \frac{\lambda_0^3 \delta'}{2c^2}} \leq \dfrac{1}{4}\left(\frac{\lambda_0^4}{4Cc^2}\right)^{\frac{1}{2}}= \dfrac{M_2}{4}. 
\end{equation}
We also have, by using \eqref{eq:defM}, the fact that $x \in (0,\infty) \mapsto e^{-1/x}$ is increasing and \eqref{eq:Tlambda}, for $\lambda_0$ sufficiently large,
\begin{multline*}
\tilde{C}_1(t_2-t_1)(M^2_2 +e^{- \frac{\delta'}{\lambda(T_\lambda - t_1)}}) M_2  = \tilde{C}T_\lambda^2\left( \dfrac{\lambda_0^4}{4\tilde{C}c^2} +e^{- \frac{\delta'}{2\lambda T_\lambda^2 }}\right) M_2 \\
 \leq  \dfrac{\tilde{C}c^2}{\lambda_0^4}  \left( \dfrac{\lambda_0^4}{4\tilde{C}c^2} +e^{- \frac{\lambda_0^3 \delta'}{2c^2}} \right) M_2 
\leq \left( \dfrac{1}{4} +\dfrac{\tilde{C}c^2}{\lambda_0^4} e^{- \frac{\lambda_0^3 \delta'}{2c^2}} \right) M_2 
 \leq \dfrac{M_2}{2}.
\end{multline*}
Collecting this previous estimate with \eqref{eq:M1pluspetitM2} and \eqref{estim:g21}, recalling \eqref{eq:GammaSameBall}, we then obtain that \eqref{bound:gamma2} holds. 

The proof of \eqref{eq:contractionprop} goes as follows
\begin{multline*}
\tilde{C}_2(t_2-t_1) \left( e^{- \frac{ \delta' }{\lambda(T_\lambda - t_1)}}+ 2M_2^{2} \right)  =  \tilde{C}T_\lambda^2 \left( e^{- \frac{ \delta' }{2\lambda T_\lambda^2 }}+ 2\dfrac{\lambda_0^4}{4\tilde{C}c^2} \right)\\
 \leq  \dfrac{\tilde{C}c^2}{\lambda_0^4} \left( e^{- \frac{\lambda_0^3 \delta'}{2c^2}}+ \dfrac{\lambda_0^4}{2\tilde{C}c^2} \right)=  \dfrac{\tilde{C}c^2}{\lambda_0^4}  e^{- \frac{\lambda_0^3 \delta'}{2c^2}}+\dfrac{1}{2} \leq \frac{3}{4},
\end{multline*}
for $\lambda_0$ sufficiently large. This conludes the proof of the claim.
\end{proof}

\textbf{Conclusion of Steps 1, 2, 3.} We have proved that $\Gamma_2$ has a unique fixed-point in $X_{M_2}$, and that this solution comes from a previous fixed point $\Gamma_1$ in $X_{M_1}$. From now on, we call $w$ the fixed point from $\Gamma_1$ and $\Gamma_2$ and fix $\psi$ as defined in \eqref{eq:ansatz}. At the end of this step, we then have constructed a solution $\psi \in C([0,t_2];H^2(\Omega) \cap H_0^1(\Omega))$ of \eqref{eq:L2CriticalNLSt2} as the time contatenation of the two solutions constructed in $[0,t_1]$ and in $[t_1,t_2]$ associated to the time contatenation of the two feedback controls defined in \eqref{eq:defK1} and \eqref{eq:defcontrolfeedback}. It is worth mentioning that the parameter $\delta$ of perturbation has to be chosen so that $\delta < M_1/2 \leq M_2/8$.  In the next step, we will choose $M_2$ small so that $2M_2< \varepsilon$. By looking at the definitions of $M_1$ and $M_2$ that could be chosen large a priori for $\lambda_0$ sufficiently large, we will obtain that $\delta$ can be chosen of the form $\delta = \frac{ \varepsilon}{16}$.\\

\textbf{Step 4: Estimation of the whole solution at time $t_2$.} From \eqref{eq:ansatz} evaluated at time $t=t_2$, we have
\begin{equation}
\psi(t_2, \cdot)=e^{-\mu (t_2-t_1)}\phi_\lambda(t_2, \cdot) + w(t_2, \cdot).
\end{equation}

Then from the estimate \eqref{eq:philambdatH2} of \Cref{lem:philambda}, the bound on $w$ given by \eqref{eq:spacefixedpoint} and the choice of $\mu$ in \eqref{eq:defmut1t2}, we have the following estimates
\begin{align*}
\| \psi(t_2, \cdot)\|_{H^2(\Omega)} &\leq    e^{-\mu (t_2-t_1)} \| \phi_\lambda(t_2, \cdot)\|_{H^2(\Omega)} +  \|w(t_2, \cdot)\|_{H^2( \Omega)} \\
 &\leq   \dfrac{C  e^{-\mu (t_2-t_1)}}{\lambda^{\alpha}(T_\lambda-t_2)^\beta} +  M_2\\
 &\leq  \dfrac{C  e^{-\frac{1}{\lambda(T_\lambda -t_2)}}}{\lambda^{\alpha}(T_\lambda-t_2)^\beta} + M_2.
\end{align*}
Recalling the choice of $M_2$ in \eqref{eq:defM} and those of $t_2$ in \eqref{eq:deft1t2}, one can choose $M_2$ sufficiently small and uniformly in $\lambda$ to verify
\begin{equation}
\| \psi(t_2, \cdot)\|_{H^2(\Omega)}  \leq 2 M_2 < \varepsilon,
\end{equation}
for $\lambda$ sufficiently large. By \eqref{eq:defK1} and  \eqref{eq:defcontrolfeedback}, the control $v$ is then given by the time contatenation of the two feedback controls $v(t,\cdot)=\mathds{1}_{[0,t_1]}(t) \mathcal{K}_\lambda^1 (\psi(t,\cdot)) + \mathds{1}_{[t_1,t_2]}(t) \mathcal{K}_\lambda^2 (\psi(t,\cdot)) \in L^\infty((0,t_2);H^2(\Omega)\cap H^1_0(\Omega))$. This ends the proof of \Cref{pr:rapidstab}.
\end{proof}

\subsection{Extension to a global solution: end of the proof of \Cref{tm:mainresult1}}
\label{sec:endproofTh1}

In this section, we prove \Cref{tm:mainresult1}. 
\begin{proof}[Proof of \Cref{tm:mainresult1}]
Fix $\varepsilon = \|Q\|_{L^2(\R^2)}/2 >0$. We then apply \Cref{pr:rapidstab} that furnishes $t_1 < t_2 \in (0,T_{\lambda})$ and a control $v_{c} \in L^2((0,t_2);H^2(\Omega) \cap H_0^1(\Omega))$ supported in $\omega$ such that the solution $\psi_{c}$ of \eqref{eq:L2CriticalNLSt2} satisfies \eqref{eq:psipetitt2}. From \Cref{tm:globalwp}, we have that the free Schrödinger equation
\begin{equation}
	\label{eq:L2CriticalNLSt2infty}
		\left\{
			\begin{array}{ll}
				 i \partial_t \psi + \Delta \psi = - |\psi|^{2} \psi  & \text{ in }  (t_2,+\infty) \times \Omega, 
				\\
				\psi = 0 & \text{ on } (t_2,+\infty)\times \partial \Omega, 
				\\
				\psi(t_2, \cdot) = \psi_{c}(t_2, \cdot) & \text{ in } \Omega,
			\end{array}
		\right.
\end{equation}
has a unique global solution $\tilde{\psi} \in C([t_2,+\infty);H^2(\Omega) \cap H_0^1(\Omega))$. We finally set the solution 
\begin{equation}
\psi = \psi_{c}\ \text{in}\ (0,t_2),\ \psi =  \tilde{\psi}\ \text{in}\ (t_2,+\infty),\ \psi  \in C([0,+\infty);H^2(\Omega) \cap H_0^1(\Omega)),
\end{equation}
asssociated to the control
\begin{equation}
v = v_{c}\ \text{in}\ (0,t_2),\ v =0\ \text{in}\ (t_2,+\infty), \ v \in L^2([0,+\infty);H^2(\Omega) \cap H_0^1(\Omega)),
\end{equation}
that furnishes a global controlled trajectory. This concludes the proof.
\end{proof}

\subsection{Local null-controllability result: end of the proof of \Cref{tm:mainresult2}}
\label{sec:proofth2}

In this section, we prove \Cref{tm:mainresult2}. The following small-time local null-controllability of \eqref{eq:L2CriticalNLSControl} will be used.
\begin{tm}
\label{thm:localnullcontrollabilityNLS}
We suppose that Assumption \ref{ass:linschroControllable} holds. Then, for every $T>0$, there exists $\delta_T>0$ such that for every $\psi_0 \in H^2(\Omega) \cap H_0^1(\Omega)$ with $\|\psi_0\|_{H^2(\Omega) } \leq \delta_T$, there exists $v \in C([0,T];H^2(\omega))$ such that the solution $\psi$ of \eqref{eq:L2CriticalNLSControl} satisfies $\psi(T,\cdot) = 0$.
\end{tm}

The proof of \Cref{thm:localnullcontrollabilityNLS} is based on a perturbative argument due to Zuazua in the context of the wave equation, see \cite{Zua90}. Roughly speaking, the small-time local null-controllability of \eqref{eq:L2CriticalNLSControl} comes from the small-time (global) null-controllability of the linearized equation \eqref{eq:linschro}. Local exact controllability in $H^1(\T)$ for \eqref{eq:L2CriticalNLSControl} has been first obtained in \cite{ILT03}. It has then been extended in $H^s(\T)$ in \cite{RZ09} for every $s \geq 0$ by using moments theory and Bourgain analysis for the treatment of the semilinearity seen as a small perturbation of the linear case. See also \cite{Lau10b}. This type of result has been generalized to any $d$-dimensional torus $\T^d$, $d \geq 2$, in \cite{RZ10}. In \cite{DGL06}, then in \cite{Lau10}, the authors obtain the same kind of local null-controllability results but for defocusing nonlinear Schrödinger equations on manifolds without boundary. They are only considering such cases because their goal is to obtain global controllability results by using first a stabilization procedure. See also \cite{LBM23}. In particular, the nonnegative sign of the nonlinear energy in the defocusing case is crucially used. For a survey of these results up to 2014, one can read \cite{Lau14}. We decide to postpone the proof of \Cref{thm:localnullcontrollabilityNLS} in \Cref{sec:app} for the sake of completeness.

\begin{proof}[Proof of \Cref{tm:mainresult2}]
Fix $\varepsilon = \delta_{T/2} >0$ where $\delta_{T/2}$ is the radius of local null-controllability of \eqref{eq:L2CriticalNLSControl} for data in $H^2(\Omega) \cap H_0^1(\Omega)$ at time $T/2$. We then apply \Cref{pr:rapidstab} that furnishes $t_2 \in (0,T_{\lambda})$ and a control $v_{c} \in L^2((0,t_2);H^2(\Omega) \cap H_0^1(\Omega))$ supported in $\omega$ such that the solution $\psi_{c}$ of \eqref{eq:L2CriticalNLSt2} satisfies \eqref{eq:psipetitt2}. Indeed, up to taking $\lambda>0$ larger, $t_2=T_\lambda(1-T_\lambda) < T/2$. From \Cref{thm:localnullcontrollabilityNLS}, we have that the free nonlinear Schrödinger equation
\begin{equation}
	\label{eq:L2CriticalNLSt2T}
		\left\{
			\begin{array}{ll}
				 i \partial_t \psi + \Delta \psi = - |\psi|^{2} \psi + v 1_{\omega} & \text{ in }  (t_2,T) \times \Omega, 
				\\
				\psi = 0 & \text{ on } (t_2,T)\times \partial \Omega, 
				\\
				\psi(t_2, \cdot) = \psi_{c}(t_2, \cdot) & \text{ in } \Omega,
			\end{array}
		\right.
\end{equation}
is null-controllable. There exists $v_l \in C([t_2,T];H^2(\omega))$ such that the solution $\tilde{\psi}$ of \eqref{eq:L2CriticalNLSt2T} satisfies $\tilde{\psi}(T,\cdot) = 0$. We finally set the solution 
\begin{equation} 
\psi = \psi_{c}\ \text{in}\ (0,t_2),\ \psi =  \tilde{\psi}\ \text{in}\ (t_2,T),\ \psi  \in C([0,T];H^2(\Omega) \cap H_0^1(\Omega)),
\end{equation}
asssociated to the control,
\begin{equation}
v = v_{c}\ \text{in}\ (0,t_2),\ v =v_{l} \ \text{in}\ (t_2,T), \ v \in L^{\infty}([0,T];H^2(\Omega) \cap H_0^1(\Omega)),
\end{equation}
that furnishes a global controlled trajectory satisfying \eqref{eq:psizero}. This concludes the proof.
\end{proof}

\subsection{Proofs of the results in the $3$-d case}
\label{sec:3dextension}

Let us now present the main differences in the proofs of \Cref{tm:mainresult13d} and \Cref{tm:mainresult23d}. \\

\textbf{Kato's method for the rapid stabilization.} The first difficulty appears in the proof of Step 1 and Step 2 of the rapid stabilization result i.e. \Cref{pr:rapidstab}.

\begin{proof}[Sketch of the proof of Step 1 and Step 2 of \Cref{pr:rapidstab}]
To prove that $\Gamma_1$ and $\Gamma_2$ have a unique fixed-point in $X_{M_1}$ and $X_{M_2}$ respectively (simply denoted below $X_M$ for sake of simplicity), we use Kato's method \cite{Kat87}. We endow $X_M$ defined by \eqref{eq:spacefixedpoint} with the distance
\[
d(u,v)= \sup_{t \in [t_1,t_2]} \|u-v\|_{L^2(\Omega)}. 
\]
The following lemma holds.
\begin{lm}[\cite{God11}] 
The space $X_{M}$ endowed with the distance $d$ is a complete metric space.  
\end{lm}
Moreover, using the embedding $H^2(\Omega \setminus \omega_1) \hookrightarrow L^\infty(\Omega \setminus \omega_1)$, we have the following estimates.
\begin{lm}[Lemma 1 of \cite{God11}]\label{lm:cubicH2}
There exists $C>0$ such that for every $u,v\in H^2(\Omega)$, 
\begin{itemize}
\item $ \| |u|^{4/3}u-|v|^{4/3}v \|_{L^2(\Omega)} \leq C \|u-v\|_{L^2(\Omega)} \Big(\|u\|_{L^\infty(\Omega)} + \|v\|_{L^\infty(\Omega)}\Big)^{4/3}, $ 
\item $ \| |u|^{4/3}u-|v|^{4/3}v \|_{H^1(\Omega)} \leq C \|u-v\|_{H^2(\Omega)} \Big(\|u\|_{H^2(\Omega)} + \|v\|_{H^2(\Omega)}\Big)^{4/3}, $ 
\item $ \| |u|^{4/3}u \|_{H^2(\Omega)} \leq C \|u\|_{H^2(\Omega)}^{1+4/3}$.
\end{itemize}
\end{lm} 

Let us now proceed to the estimates of $N_1$ and $N_2$ of Step 2, the estimates of Step 1 being similar. By using the decomposition of $\phi_{\lambda}$ in \eqref{eq:decompositionphilambda}, the estimate on $r_{\lambda}$ in \eqref{eq:estimationrlambda}, the estimate \eqref{eq:boundRLambdaOutsideomega} of $R_{\lambda}$ in Lemma \ref{lem:R}, the estimates \eqref{eq:outsideomega1H2} and Lemma \ref{lm:cubicH2}, we have the following sequence of estimates
\begin{align}
\left\| \int_{t_1}^{t} S(t-\tau)N_1(\phi_\lambda)(\tau) d\tau   \right\|_{X} \leq &  \int_{t_1}^{t_2} \| N_1(\phi_\lambda)(\tau) \|_{H^2(\Omega)} d\tau \notag\\
\leq &  C \int_{t_1}^{t_2} \| N_1(\phi_\lambda)(\tau) \|_{H^2(\Omega \setminus \omega_1)} d\tau  \notag \\
\leq & C (1+\mu) \int_{t_1}^{t_2} \|\phi_\lambda(\tau)\|_{H^2(\Omega \setminus \omega_1)}+ \|\phi_\lambda(\tau)\|_{H^2(\Omega \setminus \omega_1)}^{1+4/3} d\tau \notag \\
\leq  & C (1+\mu) \int_{t_1}^{t_2} e^{- \frac{\delta'}{\lambda(T_\lambda - \tau)}} d \tau  \notag\\ 
\leq  & C(1+\mu)(t_2-t_1) e^{- \frac{\delta'}{\lambda(T_\lambda - t_1)}}\qquad \forall t \in (t_1,t_2),
\label{eq:estimateN13d}
\end{align}
for some positive constant $C>0$ and $\delta'=\delta'(r, \kappa)>0$. To estimate $N_2$ defined in \eqref{eq:defN2}, we use Lemma \ref{lm:cubicH2}, the fact that $x\in \R^+ \mapsto x^{1+4/3}$ is a convex function and the decomposition \eqref{eq:ansatz} of $\psi$,
\begin{align*}
\| (1-\chi)(|\psi|^{4/3}\psi-&e^{-(1+4/3)\mu (t-t_1)}|\phi_\lambda|^{4/3} \phi_\lambda) \|_{H^2(\Omega \setminus \omega)} \\
&\leq C \| |\psi|^{4/3}\psi-e^{-(1+4/3)\mu (t-t_1)}|\phi_\lambda|^{4/3} \phi_\lambda \|_{H^2(\Omega \setminus \omega)} \\
&\leq C (\| |\psi|^{4/3}\psi \|_{H^2(\Omega \setminus \omega)} + \| e^{-(1+4/3)\mu (t-t_1)}|\phi_\lambda|^{4/3} \phi_\lambda \|_{H^2(\Omega \setminus \omega)}) \\
&\leq C (\| |\psi|^{4/3}\psi \|_{H^2(\Omega \setminus \omega)} +  e^{-(1+4/3)\mu (t-t_1)} \||\phi_\lambda|^{4/3} \phi_\lambda \|_{H^2(\Omega \setminus \omega)}) \\
&\leq C (\| \psi \|^{1+4/3}_{H^2(\Omega \setminus \omega)} +  e^{-(1+4/3)\mu (t-t_1)} \| \phi_\lambda \|^{1+4/3}_{H^2(\Omega \setminus \omega)}) \\
&\leq C (\| e^{-\mu (t-t_1)}\phi_\lambda + w \|^{1+4/3}_{H^2(\Omega \setminus \omega)} +  e^{-(1+4/3)\mu (t-t_1)} \| \phi_\lambda \|^{1+4/3}_{H^2(\Omega \setminus \omega)}) \\
&\leq C ((\| e^{-\mu (t-t_1)}\phi_\lambda \|_{H^2(\Omega \setminus \omega)}+ \|w \|_{H^2(\Omega \setminus \omega)})^{1+4/3} +  e^{-(1+4/3)\mu (t-t_1)} \| \phi_\lambda \|^{1+4/3}_{H^2(\Omega \setminus \omega)})\\
&\leq C ( e^{-(1+4/3)\mu (t-t_1)}\| \phi_\lambda \|_{H^2(\Omega \setminus \omega)}^{1+4/3} + \|w \|_{H^2(\Omega \setminus \omega)}^{1+4/3} +  e^{-(1+4/3)\mu (t-t_1)} \| \phi_\lambda \|^{1+4/3}_{H^2(\Omega \setminus \omega)})\\
&\leq C (  \|w \|_{H^2(\Omega \setminus \omega)}^{1+4/3}+  e^{-(1+4/3)\mu (t-t_1)} \| \phi_\lambda \|^{1+4/3}_{H^2(\Omega \setminus \omega)}),
\end{align*}
for a constant $C>0$ depending only on $\Omega$ and $\omega$.
Hence, 
\begin{align}
\left\| \int_{t_1}^{t} S(t-\tau)N_2(w,\phi_\lambda)(\tau) d\tau   \right\|_{X}  & \leq  C \int_{t_1}^{t_2} \| N_2(w,\phi_\lambda)(\tau) \|_{H^2(\Omega \setminus \omega_1)}  d\tau \notag\\
& \leq C \int_{t_1}^{t_2} \left( \| w(\tau, \cdot) \|_{H^2(\Omega \setminus \omega_1)}^{1+4/3} +  \| \phi_\lambda(\tau,\cdot) \|_{H^2(\Omega \setminus \omega_1)}^{1+4/3} \right)   d\tau \notag \\
 & \leq C(t_2-t_1) (M^{1+4/3} + e^{- \frac{ \delta'}{\lambda(T_\lambda - t_1)}} ) \qquad \forall t \in (t_1,t_2), \label{eq:estimateN23d}
\end{align}
for some positive constant $C>0$ and $\delta'=\delta'(r, \kappa)>0$. Hence, to prove that $\Gamma_{i} : X_M \to X_M, i=1,2$, we follow the same procedure with the use of \eqref{eq:estimateN13d} and \eqref{eq:estimateN23d}.

Let us now prove the contraction property. Denote $\psi_w = \Gamma(w)$ and $\psi_z = \Gamma(z)$ associated to $w,z\in X_M$ respectively. For any  $w,z\in X_M$, we have, using Lemma \ref{lm:cubicH2} and the Sobolev embedding $H^2(\Omega \setminus \omega_1)\hookrightarrow L^\infty(\Omega \setminus \omega_1)$,
\begin{align*}
d(\Gamma(w),\Gamma(z)) & \leq \int_{t_1}^{t_2} \| (1-\chi)(|\psi_w|^{4/3}\psi_w - |\psi_z|^{4/3}\psi_z) \|_{L^2(\Omega \setminus \omega_1)} d \tau \\
& \leq C \int_{t_1}^{t_2} \| |\psi_w|^{4/3}\psi_w - |\psi_z|^{4/3}\psi_z \|_{L^2(\Omega \setminus \omega_1)} d \tau \\
& \leq C \int_{t_1}^{t_2} \left( \| \psi_w \|_{H^2(\Omega \setminus \omega_1)}^{4/3} + \|\psi_z \|_{H^2(\Omega \setminus \omega_1)}^{4/3} \right) \|\psi_w-\psi_z\|_{L^2(\Omega \setminus \omega_1)} d \tau \\
& \leq C \int_{t_1}^{t_2} \left( \| \phi_\lambda \|_{H^2(\Omega \setminus \omega_1)}^{4/3} + \|w \|_{H^2(\Omega \setminus \omega_1)}^{4/3} + \|z\|_{H^2(\Omega \setminus \omega_1)}^{4/3} \right) \|w - z\|_{L^2(\Omega \setminus \omega_1)} d \tau \\
& \leq C(t_1-t_2) \left( e^{- \frac{ \delta'}{\lambda(T_\lambda - t_1)}}+ 2M^{4/3} \right) d(w,z). 
\end{align*}
This last estimate yields the contraction property.
\end{proof}

\textbf{Extension to a global solution for small data in $H^2(\Omega) \cap H_0^1(\Omega)$.} The second difficulty is the extension of $\psi$ to a global solution in $[t_2, +\infty)$ due to the lack global well-posedness in $H^2(\Omega) \cap H_0^1(\Omega)$ in $3$-d. However, by using the following $3$-d Gagliardo-Nirenberg inequality
\begin{equation}
\displaystyle E(\psi)  \geq \frac{1}{2} \left(\int_{\Omega} |\nabla \psi|^2\right)\left(1 - \left( \dfrac{\|\psi\|_{L^2(\Omega)}}{\|Q\|_{L^2(\R^3)}} \right)^{4/d} \right),
\end{equation}
and the local well-posedness of \eqref{eq:L2CriticalNLS3d} in $H^2(\Omega) \cap H_0^1(\Omega)$, we have the following result.

\begin{tm}
For every initial data $\psi_0 \in  H^2(\Omega) \cap H_0^1(\Omega)$ satisfying $\|\psi_0\|_{L^2(\Omega)} < \|Q\|_{L^2(\R^3)}$, \eqref{eq:L2CriticalNLS3d} admits a global solution $\psi \in C([0,+\infty);H_0^1(\Omega))$.
\end{tm}
Note that in the above result, the uniqueness is not ensured a priori because the local well-posedness of \eqref{eq:L2CriticalNLS3d} for data in $H_0^1(\Omega)$ is not guaranteed a priori.\\

\textbf{Kato's method for the local null-controllability result.} The last difficulty appears in the proof of the local null-controllabilty result i.e. \Cref{thm:localnullcontrollabilityNLS}. Indeed, because the nonlinearity $|\psi|^{4/3} \psi$ is not algebraic, to prove the contraction property in the fixed-point argument, one needs to use again Kato's method. Details are left to the reader.

\subsection{An open-loop null-control}
\label{sec:openloopcontrol}

The goal of this part is to prove the following null-controllability result, without \Cref{ass:linschroControllable}, for initial data $\psi(0,\cdot) = \phi_{\lambda}(0,\cdot)$ at time $T=T_{\lambda}$, corresponding exactly to the blow-up time of the blow-up profile $\phi_{\lambda}$.

\begin{tm}\label{tm:mainresult1openloop}
Let $\Omega$ be a smooth bounded domain in $\R^2$. Let $n \geq 1$ and $x_1, \dots, x_n$ be $n$ distinct points in $\Omega$. Let $\omega \subset \Omega$ be an open set such that $x_1, \dots, x_n \in \omega$. Then, there exists $\lambda_0 >0$ such that, for every $\lambda \geq \lambda_0$, there exists a control $v_{\lambda} \in L^{\infty}([0,T_{\lambda}];H^2(\Omega) \cap H_0^1(\Omega))$ supported in $(0,T_{\lambda}) \times \omega$ and an associated global solution $\psi \in C([0,T_{\lambda}]);H^2(\Omega) \cap H_0^1(\Omega))$ of \eqref{eq:L2CriticalNLSControl} starting from $\psi(0, \cdot) = \phi_{\lambda}(0,\cdot)$, where $\phi_\lambda$ is the blow-up profile from Theorem \ref{thm:mainresultgod}, and such that
$$ \psi(T_{\lambda},\cdot) = 0.$$
\end{tm}

\begin{proof}
Let $\theta_{\lambda}(t) \in C^{\infty}([0, T_{\lambda}];[0,1])$ be such that $\theta_{\lambda}(t) \equiv 1$ in $[T_{\lambda}/2, T_{\lambda}]$ and $\chi(x) \in C^{\infty}(\Omega;[0,1])$ such that
\begin{equation}
\label{eq:defchi}
\chi = 
		\left\{
			\begin{array}{ll}
				  1  & \text{ in }  \omega_1, 
				\\
				0 & \text{ in } \Omega \setminus \omega.
			\end{array}
		\right.
\end{equation}
Let us define
\begin{equation}
\label{eq:defpsiopenloop}
\psi(t,x) = (1- \theta(t) \chi(x)) \phi_{\lambda}(t,x)\qquad (t,x) \in (0,T_{\lambda})\times\Omega.
\end{equation} 
We look at the equation satisfied by $\psi$, we have
\begin{equation}
	\label{eq:L2CriticalNLSControlProofOp1}
		\left\{
			\begin{array}{ll}
				 i \partial_t \psi + \Delta \psi = f  & \text{ in }  (0,T_{\lambda}) \times \Omega, 
				\\
				\psi = 0 & \text{ on } (0,T_{\lambda})\times \partial \Omega, 
				\\
				\psi(0, \cdot) = \phi_0 & \text{ in } \Omega,
			\end{array}
		\right.
\end{equation}
where
\begin{align}
f &= -i \theta' \chi \phi_{\lambda} - 2 \theta \nabla \chi \cdot \nabla \phi_{\lambda} - \theta( \Delta \chi)  \phi_{\lambda} - (1-\theta\chi) |\phi_{\lambda}|^2 \phi_{\lambda}\notag\\
& = -i \theta' \chi \phi_{\lambda} - 2 \theta \nabla \chi \cdot \nabla \phi_{\lambda} - \theta( \Delta \chi)  \phi_{\lambda} - |\psi|^2 \psi 
 +  \left((1-\theta\chi)^3 - (1- \theta \chi)\right) |\phi_{\lambda}|^2 \phi_{\lambda}\notag\\
 &= - |\psi|^2 \psi  -i \theta' \chi \phi_{\lambda} - 2 \theta \nabla \chi \cdot \nabla \phi_{\lambda} - \theta( \Delta \chi)  \phi_{\lambda} + \left(-2 \theta \chi + 3 \theta^2 \chi^2 -\theta^3 \chi^3 \right)  |\phi_{\lambda}|^2 \phi_{\lambda}\notag\\
& = - |\psi|^2 \psi + v, \label{eq:sourceopenloop}
\end{align}
with
\begin{equation}
v =  -i \theta' \chi \phi_{\lambda} - 2 \theta \nabla \chi \cdot \nabla \phi_{\lambda} - \theta( \Delta \chi)  \phi_{\lambda} + \left(-2 \theta \chi + 3 \theta^2 \chi^2 -\theta^3 \chi^3 \right)  |\phi_{\lambda}|^2 \phi_{\lambda}.
\label{eq:defopennloopcontrol}
\end{equation}

First, note that from the properties of $\theta$ and $\chi$, the support of $v$ is contained in $(0,T_{\lambda}) \times \omega$. Moreover, due to the decomposition of $\phi_{\lambda} = R_{\lambda} + r_{\lambda}$ recalled in \eqref{eq:decompositionphilambda}, the estimate \eqref{eq:boundRLambdaOutsideomega} on $R_\lambda$ of \Cref{lem:R} and the estimate on $r_{\lambda}$ stated in \eqref{eq:estimationrlambda}, it is rather straightforward to prove that $v \in C([0,T_{\lambda}];H^2(\omega))$.

Secondly, by using the definition of $\psi$ in \eqref{eq:defpsiopenloop}, and again the decomposition of $\phi_{\lambda} = R_{\lambda} + r_{\lambda}$ recalled in \eqref{eq:decompositionphilambda}, the estimate \eqref{eq:boundRLambdaOutsideomega} on $R_\lambda$ of \Cref{lem:R} and the estimate on $r_{\lambda}$ stated in \eqref{eq:estimationrlambda}, it is also easy to prove that for some constants $c_1>0, c_2>0$, the following estimates holds
\begin{equation}
\psi \in C([0,T_{\lambda}];H^2(\Omega)\cap H_0^1(\Omega))\ \text{and}\ \|\psi(t,\cdot)\|_{H^2(\Omega)\cap H_0^1(\Omega)} \leq c_1 e^{-\frac{c_2}{\lambda(T_{\lambda}-t)}}\quad \forall t \in [0, T_{\lambda}).
\end{equation}
In particular, we get
\begin{equation}
\psi(T_{\lambda},\cdot) = 0.
\end{equation}
This concludes the proof.
\end{proof}

\appendix

\section{Appendix}
\label{sec:app}

The goal of this part is to provide a proof of \Cref{thm:localnullcontrollabilityNLS}.

\subsection{The controlled linear system.}
Before studying the local controllability of the nonlinear equation \eqref{eq:L2CriticalNLSControl}, let us consider the linear system
\begin{equation}
	\label{eq:L2CriticalNLSControlLinear1}
		\left\{
			\begin{array}{ll}
				 i \partial_t \psi + \Delta \psi =  a^2(x) \varphi^2(t) e^{it \Delta} \psi_0 & \text{ in }  (0,T) \times \Omega, 
				\\
				\psi = 0 & \text{ on } (0,T)\times \partial \Omega, 
				\\
				\psi(T, \cdot) = 0 & \text{ in } \Omega,
			\end{array}
		\right.
\end{equation}
for $\psi_0 \in L^2(\Omega)$ and where $a(x) \in C_c^{\infty}(\omega)$, $a \neq 0$, and $\varphi(t) \in C_c^{\infty}(0,T)$.

Let us define the linear operator 
$$\begin{array}{llll}
S :& L^2(\Omega) &\longrightarrow & L^2(\Omega)\\
    & \psi_0 & \longmapsto & \psi(0,\cdot).
\end{array}$$
where $\psi$ is the mild solution of \eqref{eq:L2CriticalNLSControlLinear1}. One can easily check that $S$ is an injective continuous map. Let us highlight that the surjectivity of $S$ would lead to the exact controllability of the linear system \eqref{eq:L2CriticalNLSControlLinear1}. Thanks to the Hilbert Uniqueness Method, the question of its surjectivity is equivalent to the observability estimates
$$\exists C_{a,\varphi} >0, \forall \psi_0 \in L^2(\Omega), \quad \|\psi_0\|^2_{L^2(\Omega)} \leq C_{a, \varphi} \int_{\R} \varphi(t)^2 \|a e^{it\Delta} \psi_0\|^2_{L^2(\Omega)} dt,$$
which are known to hold from \Cref{ass:linschroControllable}. The linear map $S$ is therefore an isomorphism from $L^2(\Omega)$ to $L^2(\Omega)$. Actually, the following proposition states that $S$ is also an isomorphism from $H^2(\Omega) \cap H_0^1(\Omega)$ to $H^2(\Omega)\cap H_0^1(\Omega)$.
\begin{pr}\label{prop:control_operator}
The Sobolev space $H^2(\Omega)\cap H_0^1(\Omega)$ is $S$ invariant and $S : H^2(\Omega)\cap H_0^1(\Omega) \longrightarrow H^2(\Omega)\cap H_0^1(\Omega)$ is an isomorphism.
\end{pr}
Let us mention that the proof of \Cref{prop:control_operator} follows from \Cref{ass:linschroControllable} and the systematic method for building smooth controls starting from smooth data in  \cite{EZ10}.

\subsection{The controlled nonlinear system}

In this part, we prove the local null-controllability result stated in \Cref{thm:localnullcontrollabilityNLS}.

\begin{proof}[Proof of \Cref{thm:localnullcontrollabilityNLS}]
For $\psi_0 \in H^2(\Omega) \cap H_0^1(\Omega)$, $\|\psi_0\|_{H^2(\Omega)} \leq \varepsilon$ with $\varepsilon >0$ sufficiently small that would be chosen later,  we consider $\psi \in C([0,T];H^2(\Omega) \cap H_0^1(\Omega))$ the unique solution of 
\begin{equation}
	\label{eq:L2CriticalNLSControlLocal1}
		\left\{
			\begin{array}{ll}
				 i \partial_t \psi + \Delta \psi = - |\psi|^{2} \psi  + a^2(x) \varphi^2(t) e^{it \Delta} \psi_0 & \text{ in }  (0,T) \times \Omega, 
				\\
				\psi = 0 & \text{ on } (0,T)\times \partial \Omega, 
				\\
				\psi(T, \cdot) = 0 & \text{ in } \Omega.
			\end{array}
		\right.
\end{equation}
 We then set $\phi \in C([0,T];H^2(\Omega) \cap H_0^1(\Omega))$ be the unique solution of
\begin{equation}
	\label{eq:L2CriticalNLSControlLocal2}
		\left\{
			\begin{array}{ll}
				 i \partial_t \phi + \Delta \phi = - |\psi|^{2} \psi   & \text{ in }  (0,T) \times \Omega, 
				\\
				\phi = 0 & \text{ on } (0,T)\times \partial \Omega, 
				\\
				\phi(T, \cdot) = 0 & \text{ in } \Omega.
			\end{array}
		\right.
\end{equation}

Let us then define $L\psi_0=\psi(0,\cdot)$ and $K\psi_0=\phi(0,\cdot)$. We therefore have 
$$\forall \psi_0 \in H^2(\Omega) \cap H_0^1(\Omega), \quad L\psi_0 = K\psi_0 +S\psi_0.$$

Our goal is to show that there exists $\eta>0$ such that $B_{H^2(\Omega)}(0,\eta) \subset \text{Im} (L)$.
Notice that the equation $u_0 =L\psi_0$ is equivalent to 
$$\psi_0= S^{-1} u_0 -S^{-1} K \psi_0$$
and this question is then equivalent to find a fixed point of $$B\psi_0:=S^{-1} u_0 -S^{-1} K \psi_0,$$ for $u_0$ sufficiently small in $H^2(\Omega)$. 

Let $0<\eta, \varepsilon \leq 1$ be two small parameters to be chosen later and $u_0 \in B_{H^2(\Omega)}(0, \eta)$. Without loss of generality, we can assume $T\leq 1$.
Since $S : H^2(\Omega) \cap H_0^1(\Omega) \longrightarrow H^2(\Omega) \cap H_0^1(\Omega)$ is an isomorphism, we have that for all $\psi_0 \in H^2(\Omega) \cap H_0^1(\Omega)$,
\begin{align*}\|B\psi_0 \|_{H^2(\Omega)} & \leq C ( \|u_0\|_{H^2(\Omega)}+\|K \psi_0\|_{H^2(\Omega)})\\
& = C ( \|u_0\|_{H^2(\Omega)}+\|\phi(0, \cdot)\|_{H^2(\Omega)}).
\end{align*}
Moreover, we have 
\begin{align*}
    \|\phi(0, \cdot)\|_{H^2(\Omega)} 
    & \leq C  \||\psi|^{2} \psi \|_{C([0,T];H^2(\Omega))}\\
    & \leq C \| \psi \|_{C([0,T];H^2(\Omega))}^{3}.
\end{align*}
Furthermore, by using the fact that the flow map is Lipschitz on the bounded set $B_{H^2(\Omega) \cap H_0^1(\Omega)}(0,1) \times B_{C([0,T]; H^2(\Omega))}(0,1)$, we obtain for all $\psi_0 \in \overline{B_{H^2(\Omega) \cap H_0^1(\Omega)}}(0, \varepsilon)$, 
$$\|\psi\|_{C([0,T];H^2(\Omega))} \leq C \|\psi_0\|_{H^2(\Omega)} \leq C \varepsilon.$$
As a consequence, we deduce that for all $\psi_0 \in \overline{B_{H^2(\Omega)}(0, \varepsilon)}$ with $\psi_0 \in H^2(\Omega) \cap H_0^1(\Omega)$,
$$\|B\psi_0\|_{H^2(\Omega)} \leq C (\eta + \varepsilon^3),$$
for some positive constant $C>0$ independent on $\varepsilon$ and $\eta$. We can therefore choose $\varepsilon_0>0$ such that 
$$\varepsilon^3_0 \leq \frac{\varepsilon_0}{2C},$$
and $\eta= \frac{\varepsilon_0}{2C}$,
and we obtain that the closed ball $\overline{B_{H^2(\Omega) \cap H_0^1(\Omega)}}(0, \varepsilon_0)$ is $B$ invariant. It remains to check that $B$ is a contraction mapping on this ball.
Let $\psi_0, \psi_1 \in \overline{B_{H^2(\Omega) \cap H_0^1(\Omega)}}(0, \varepsilon_0)$ and $\psi_{\psi_0}, \psi_{\psi_1}$ be the associated solutions. We have
\begin{align*}
\|B\psi_0-B\psi_1\|_{H^2(\Omega)} & = \|S^{-1}(K\psi_0-K\psi_1)\|_{H^1(\T^d)}  \\
& \leq C \|\phi_0(0, \cdot)-\phi_1(0, \cdot)\|_{H^2(\Omega)}\\
& \leq C\|\phi_0-\phi_1\|_{C([0,T];H^2(\Omega))} \\
& \leq C \|\psi_{\psi_0}|^{2} \psi_{\psi_0}-|\psi_{\psi_1}|^{2} \psi_{\psi_1}\|_{C([0,T];H^2(\Omega))},
\end{align*}
Then, 
\begin{align*}\|B\psi_0 -B\psi_1\|_{H^2(\Omega)} & \leq C \left(\|\psi_{\psi_0}\|^{2}_{C([0,T];H^2(\Omega))}+\|\psi_{\psi_1}\|^{2}_{C([0,T];H^2(\Omega))}\right) \|\psi_{\psi_0}- \psi_{\psi_1}\|_{C([0,T];H^2(\Omega))}.
\end{align*}
By using once again the fact that the flow map is Lipschitz on bounded set, we obtain a new constant $C>0$ independent on $\varepsilon_0$ such that
\begin{align*}\|B\psi_0- B\psi_1\|_{H^2(\Omega)} &\leq C \left(\|\psi_0\|^{2}_{H^2(\Omega)}+ \|\psi_1\|^{2}_{H^2(\Omega)} \right) \|\psi_0-\psi_1\|_{H^2(\Omega)}\\
& \leq 2C\varepsilon^{2}_0 \|\psi_0-\psi_1\|_{H^2(\Omega)}.
\end{align*}
By now, we set $\varepsilon_0 \leq \frac1{2\sqrt C}$ and $B : \overline{B_{H^2(\Omega) \cap H_0^1(\Omega)}}(0, \varepsilon_0) \longrightarrow \overline{B_{H^2(\Omega) \cap H_0^1(\Omega)}}(0, \varepsilon_0)$ is a contraction mapping and admits an unique fixed point, according to the Banach fixed-point theorem.
\end{proof}

\bigskip

{\noindent \bf Acknowledgements.} The authors would like to thank Hoai-Minh Nguyen for stimulating discussions about \Cref{sec:openloopcontrol}.

\bibliographystyle{alpha}
\small{\bibliography{NLSchrodinger}}

\end{document}